\documentclass{amsart}%
\usepackage{amsfonts}
\usepackage{amsmath}
\usepackage{amssymb}
\usepackage{graphicx}%
\newtheorem{theorem}{Theorem}
\theoremstyle{plain}

\newtheorem{definition}{Definition}
\newtheorem{example}{Example}

\newtheorem{lemma}{Lemma}

\newtheorem{remark}{Remark}

\numberwithin{equation}{section}
\begin{document}
\title[Local Zeta Functions for Rational Functions]{Local Zeta Functions for Rational Functions and Newton Polyhedra}
\author{Miriam Bocardo-Gaspar}
\address{Centro de Investigacion y de Estudios Avanzados del I.P.N., Departamento de
Matematicas, Av. Instituto Politecnico Nacional 2508, Col. San Pedro
Zacatenco, Mexico D.F., C.P. 07360, Mexico}
\email{mbocardo@math.cinvestav.mx}
\author{W. A. Z\'{u}\~{n}iga-Galindo}
\address{Centro de Investigacion y de Estudios Avanzados del I.P.N., Departamento de
Matematicas, Av. Instituto Politecnico Nacional 2508, Col. San Pedro
Zacatenco, Mexico D.F., C.P. 07360, Mexico}
\email{wazuniga@math.cinvestav.edu.mx}
\thanks{The second author was partially supported by Conacyt Grant No. 250845.}
\subjclass[2000]{Primary 14G10, 11S40; Secondary 14M25}
\keywords{Igusa local zeta functions, Newton polyhedra, non-degeneracy conditions.}

\begin{abstract}
In this article, we introduce a notion of non-degeneracy, with respect to
certain Newton polyhedra, for rational functions over non-Archimedean locals
fields of arbitrary characteristic. We study the local zeta functions attached
to non-degenerate rational functions, we show the existence of a meromorphic
continuation for these zeta functions,\ as rational functions of $q^{-s}$, and
give explicit formulas. In contrast with the classical local zeta functions,
the meromorphic continuation of zeta functions for rational functions have
poles with positive and negative real parts.

\end{abstract}
\maketitle

\section{Introduction}

The local zeta functions in the Archimedean setting, i.e. in $\mathbb{R}$ or
$\mathbb{C}$, were introduced in the 50's by I. M. Gel'fand and G. E. Shilov
\cite{Gelfand}. An important motivation was that the meromorphic continuation
for the local zeta functions implies the existence of fundamental solutions
for differential operators with constant coefficients. The meromorphic
continuation was established, independently, by M. Atiyah \cite{Atiyah} and J.
Bernstein \cite{Ber}. On the other hand, by the middle of the 60's, A. Weil
studied local zeta functions, in the Archimedean and non-Archimedean settings,
in connection with the Poisson-Siegel formula \cite{Weil}. In the 70's, using
Hironaka's resolution of singularities theorem, J.-I. Igusa \ developed a
uniform theory for local zeta functions and oscillatory integrals attached to
polynomials with coefficients in a field of characteristic zero \cite{Igusa1},
\cite{Igusa}. In the $p$-adic setting, local zeta functions are connected with
the number of solutions of polynomial congruences $\operatorname{mod}$ $p^{m}$
and with exponential sums $\operatorname{mod}$ $p^{m}$. In addition, there are
many intriguing conjectures relating the poles of the local zeta functions
with topology of complex singularities, see e.g. \cite{Denef}, \cite{Igusa}.
More recently, J. Denef and F. Loeser introduced in \cite{D-L} the motivic
zeta functions which constitute a vast generalization of the $p$-adic local
zeta functions.

In \cite{V-Z2}\ W. Veys and W. A. Z\'{u}\~{n}iga-Galindo extended Igusa's
theory to the case of rational functions, or, more generally, meromorphic
functions $f/g$, with coefficients in a local field of characteristic zero.
This generalization is far from being straightforward due to the fact that
several new geometric phenomena appear. Also, the oscillatory integrals have
two different asymptotic expansions: the `usual'\ one when the norm of the
parameter tends to infinity, and another one when the norm of the parameter
tends to zero. The first asymptotic expansion is controlled by the poles (with
negative real parts) of all the twisted local zeta functions associated to the
meromorphic functions $f/g-c$, for certain special values $c$. The second
expansion is controlled by the poles (with positive real parts) of all the
twisted local zeta functions associated to $f/g$. There are several
mathematical and physical motivations for studying these new local zeta
functions. For instance, S. Gusein-Zade, I. Luengo and A. Melle-Hern\'{a}ndez
have studied the complex monodromy (and A'Campo zeta functions attached to it)
of meromorphic functions, see e.g. \cite{G-LM}, \cite{G-LM1}, \cite{G-LM2}.
This work drives naturally to ask about the existence of local zeta functions
with poles related with the monodromies studied by the mentioned authors. From
a physical perspective, the local zeta functions attached to meromorphic
functions are very alike to parametric Feynman integrals and to $p$-adic
string amplitudes, see e.g. \cite{Be-Bro}, \cite{Bo-We}, \cite{B-F-O-W},
\cite{Marcolli}. For instance in \cite[Section 3.15]{Marcolli}, M. Marcolli
pointed out explicitly that the motivic Igusa zeta function constructed by J.
Denef and F. Loeser may provide the right tool for a motivic formulation of
the dimensionally regularized parametric Feynman integrals.

This article aims to study the local zeta functions attached to a rational
function $f/g$ with coefficients in a local field of arbitrary characteristic,
when $f/g$ is non-degenerate with respect to a certain Newton polyhedron. In
\cite{Leon} E. Le\'{o}n-Cardenal and W. A. Z\'{u}\~{n}iga-Galindo studied
similar matters. In this article, we present a more suitable and general
notion of non-degeneracy which allows us to study the local zeta functions
attached to much larger class of rational functions. Our article is organized
as follows. In Section \ref{Sect_1} we summarize some basic aspects about
non-Archimedean local fields and compute some $\pi$-adic integrals that are
needed in the article. In Section \ref{Sect_3} we review some basic aspects
about polyhedral subdivisions and Newton polyhedra, we also introduce a notion
of non-degeneracy for polynomials mappings. It seems that our notion of
non-degeneracy is a new one. In Section \ref{Sect_4} we study the meromorphic
continuation for multivariate local zeta functions attached to non-degenerate
mappings. These local zeta functions were introduced by L. Loeser in
\cite{Loeser}. We give a very general geometric description of the poles of
the meromorphic continuation of these functions, see Theorem \ref{Theorem1}.
Our results extend some of the well-known results due to Hoornaert and Denef
\cite{DenefHoor}, and Bories \cite{Bories}. In Section \ref{Sect_5} we study
the local zeta functions attached to rational functions satisfying a suitable
non-degeneracy condition. In Theorem \ref{Theorem2}, we give a geometric
description of the poles of the meromorphic continuation of these functions.
The real parts of the poles of the meromorphic continuation of these functions
are positive and negative rational numbers. Finally, in Section \ref{Sect_6},
we describe the `smallest positive and largest negative poles' appearing in
the meromorphic continuation of these new local zeta functions, see Theorems
\ref{Theorem3}, \ref{Theorem4}.

\textbf{Acknowledgement.} The authors wish to thank the referee for his/her
careful reading of the original manuscript.

\section{\label{Sect_1}Preliminaries}

In this article we work with a non-discrete locally compact field $K$ of
arbitrary characteristic. We will say that a such field is a
\textit{non-Archimedean local field} of arbitrary characteristic. By a
well-known classification theorem, a non-Archimedean local field is a finite
extension of $\mathbb{Q}_{p}$, the field of $p$-adic numbers, or is
\textit{the field of formal Laurent series} $\mathbb{F}_{q}\left(  \left(
T\right)  \right)  $ over a finite field $\mathbb{F}_{q}$. In the first case
we say that $K$ is \textit{a} $p$\textit{-adic field}. For further details the
reader may consult \cite[Chapter 1]{WeilB}.

Let $K$ be a non-Archimedean local field of arbitrary characteristic and let
$\mathcal{O}_{K}$ be the ring of integers of $K$ and let the residue field of
$K$ be $\mathbb{F}_{q}$, the finite field with $q=p^{m}$ elements, where $p$
is a prime number. For $z\in K\smallsetminus\left\{  0\right\}  $, let
$ord(z)\in\mathbb{Z}\cup\left\{  +\infty\right\}  $ denote \textit{the
valuation} of $z$, let $|z|_{K}=q^{-ord(z)}$ denote the normalized
\textit{absolute value} (or \textit{norm}), \ and let $ac(z)=z\pi^{-ord(z)}$
denote the \textit{angular component}, where $\pi$ is a fixed uniformizing
parameter of $K$. We extend the norm $|\cdot|_{K}$ to $K^{n}$ by taking
$||(x_{1},\ldots,x_{n})||_{K}:=\max\left\{  |x_{1}|_{K},\ldots,|x_{n}%
|_{K}\right\}  $. Then $(K^{n},||\cdot||_{K})$ is a complete metric space and
the metric topology is equal to the product topology.

Along this paper, vectors will be written in boldface, so for instance we will
write $\mathbf{b}:=(b_{1},\ldots,b_{l})$ where $l$ is a positive integer. For
polynomials we will use $\boldsymbol{x}=(x_{1},\ldots,x_{n})$, thus $h\left(
\boldsymbol{x}\right)  =$ $h(x_{1},\ldots,x_{n})$. For each $n$-tuple of
natural numbers $\boldsymbol{k}=(k_{1},\ldots,k_{n})\in\mathbb{N}^{n}$, we
will denote by $\sigma(\boldsymbol{k})$ the sum of all its components i.e.
$\sigma(\boldsymbol{k})=k_{1}+k_{2}+\ldots+k_{n}$. Furthermore, we will use
the notation $|d\boldsymbol{x}|_{K}$ for the Haar measure on $\left(
K^{n},+\right)  $ normalized so that the measure of $\mathcal{O}_{K}^{n}$ is
equal to one. In dimension one, we will use the notation $|dx|_{K}$.

\subsection{\label{Sect_2}Multivariate local zeta functions}

We denote by $\mathcal{S}(K^{n})$ the $\mathbb{C}$-vector space consisting of
all $\mathbb{C}$-valued locally constant functions over $K^{n}$ with compact
support. An element of $\mathcal{S}(K^{n})$ is called a\textit{
Bruhat-Schwartz function} or a \textit{test function}. Along this article we
work with a polynomial mapping $\boldsymbol{h}=(h_{1},\ldots,h_{r}%
):K^{n}\rightarrow K^{r}$ such that each $h_{i}(\boldsymbol{x})$ is a
non-constant polynomial in $\mathcal{O}_{K}[x_{1},\ldots,x_{n}]\backslash
\pi\mathcal{O}_{K}[x_{1},\ldots,x_{n}]$ and $r\leq n$. Let $\Phi$ a
Bruhat-Schwartz function and let $\boldsymbol{s}=(s_{1},\ldots,s_{r}%
)\in\mathbb{C}^{r}$. The local zeta function associated to $\Phi$ and
$\boldsymbol{h}$ is defined as
\[
Z_{\Phi}(\boldsymbol{s},\boldsymbol{h})=\int\limits_{{K}^{n}\backslash D_{K}%
}\Phi(\boldsymbol{x})\prod_{i=1}^{r}\left\vert h_{i}(\boldsymbol{x}%
)\right\vert _{K}^{s_{i}}|d\boldsymbol{x}|_{K}%
\]
for $\operatorname{Re}(s_{i})>0$ for all $i$, where $D_{K}:=\cup_{i\in\left\{
1,\ldots,r\right\}  }\left\{  \boldsymbol{x}\in K^{n};h_{i}(\boldsymbol{x}%
)=0\right\}  $. Notice that $Z_{\Phi}(\boldsymbol{s},\boldsymbol{h})$
converges for $\operatorname{Re}(s_{i})>0$ for all $i=1,\ldots,r$. If $\Phi$
is the characteristic function of $\mathcal{O}_{K}^{n}$ we use the notation
$Z(\boldsymbol{s},\boldsymbol{h})$ instead of $Z_{\Phi}(\boldsymbol{s}%
,\boldsymbol{h})$. In the case of polynomial mappings with coefficients in a
local field of characteristic zero (not necessarily non-Archimedean and
without the condition $r\leq n$), the theory of local zeta functions of type
$Z_{\Phi}(\boldsymbol{s},\boldsymbol{h})$ was established by F. Loeser in
\cite{Loeser}.

Denote by $\overline{\boldsymbol{x}}$ the image of an element of
$\mathcal{O}_{K}^{n}$ under the canonical homomorphism $\mathcal{O}_{K}%
^{n}\rightarrow\mathcal{O}_{K}^{n}/(\pi\mathcal{O}_{K})^{n}\cong\mathbb{F}%
_{q}^{n}$, we call $\overline{\boldsymbol{x}}$ \textit{the reduction modulo}
$\pi$ of $\boldsymbol{x}$. Given $h(\boldsymbol{x})\in\mathcal{O}_{K}%
[x_{1},\ldots,x_{n}]$, we denote by $\overline{h}(\boldsymbol{x})$ the
polynomial obtained by reducing modulo $\pi$ the coefficients of
$h(\boldsymbol{x})$. Furthermore if $\boldsymbol{h}=(h_{1},\ldots,h_{r})$ is a
polynomial mapping with $h_{i}\in\mathcal{O}_{K}[x_{1},\ldots,x_{n}]$ for all
$i$, then $\overline{\boldsymbol{h}}:=(\overline{h}_{1},\ldots,\overline
{h}_{r})$ denotes the polynomial mapping obtained by reducing modulo $\pi$ all
the components of $\boldsymbol{h}$.

\subsection{Some $\pi$-adic integrals}

Let $\boldsymbol{h}=(h_{1},h_{2},\ldots,h_{r})$ be a polynomial mapping as
above. For $\boldsymbol{a}\in({\mathcal{O}}_{K}^{\times})^{n}$, we set
\begin{equation}
J_{\boldsymbol{a}}(\boldsymbol{s},\boldsymbol{h}):=\int\limits_{\boldsymbol{a}%
+(\pi{\mathcal{O}}_{K})^{n}\smallsetminus D_{K}}\text{ }\prod_{i=1}^{r}%
|h_{i}(\boldsymbol{x})|_{K}^{s_{i}}|d\boldsymbol{x}|_{K}, \label{Integral_J}%
\end{equation}
$\boldsymbol{s}=(s_{1},\dots s_{r})\in\mathbb{C}^{r}$ with $\operatorname{Re}%
(s_{i})>0$, $i=1,\ldots,r$.

The Jacobian matrix of $\boldsymbol{h}$ at $\boldsymbol{a}$ is $Jac\left(
\boldsymbol{h},\boldsymbol{a}\right)  =\left[  \frac{\partial h_{i}}{\partial
x_{j}}\left(  \boldsymbol{a}\right)  \right]  _{\substack{1\leq i\leq r\\1\leq
j\leq n}}$ with $r\leq n$. In a similar way we define the Jacobian matrix of
$\overline{\boldsymbol{h}}$ at $\overline{\boldsymbol{a}}$. For every
non-empty subset $I$ of $\left\{  1,\ldots,r\right\}  $ we set $\ Jac\left(
\overline{\boldsymbol{h}}_{I},\overline{\boldsymbol{a}}\right)  :=\left[
\frac{\partial\bar{h}_{i}}{\partial x_{j}}\left(  \overline{\boldsymbol{a}%
}\right)  \right]  _{\substack{i\in I\\1\leq j\leq n}}$ .

\begin{lemma}
\label{Lemma1}Let $I$ be the subset of $\left\{  1,\ldots,r\right\}  $ such
that $\overline{h}_{i}(\overline{\boldsymbol{a}})=0\Leftrightarrow i\in I$.
\ Assume that $\boldsymbol{a}\notin D_{K}$ and that $Jac\left(  \overline
{\boldsymbol{h}}_{I},\overline{\boldsymbol{a}}\right)  $ has rank $m=Card(I)$
for $I\neq\varnothing$. Then $J_{\boldsymbol{a}}(\boldsymbol{s},\boldsymbol{h}%
)$ equals
\[
\left\{
\begin{array}
[c]{lll}%
q^{-n} & \text{if} & I=\varnothing\\
&  & \\
q^{-n}\prod\limits_{i\in I}\frac{(q-1)q^{-1-s_{i}}}{1-q^{-1-s_{i}}} &
\text{if} & I\neq\varnothing.
\end{array}
\right.
\]

\end{lemma}

\begin{proof}
By change of variables we get
\[
J_{\boldsymbol{a}}(\boldsymbol{s},\boldsymbol{h})=q^{-n}\int
\limits_{{\mathcal{O}}_{K}^{n}\smallsetminus\cup_{i\in\left\{  1,\ldots
,r\right\}  }\left\{  \boldsymbol{x}\in K^{n};h_{i}(\pi\boldsymbol{x}%
+\boldsymbol{a})=0\right\}  }\text{ }\prod_{i=1}^{r}|h_{i}(\pi\boldsymbol{x}%
+\boldsymbol{a})|_{K}^{s_{i}}|d\boldsymbol{x}|_{K}.
\]
We first consider the case $I=\varnothing$. Then $h_{i}(\boldsymbol{a}%
)\not \equiv 0\operatorname{mod}\pi$, thus $|h_{i}(\pi\boldsymbol{x}%
+\boldsymbol{a})|_{K}=1$, and $J_{\boldsymbol{a}}(\boldsymbol{s}%
,\boldsymbol{h})=q^{-n}$. In the case $I\neq\varnothing$, by reordering $I$
(if necessary) we can suppose that $I=\left\{  1,\ldots,m\right\}  $ with
$m\leq r$. Integral $J_{\boldsymbol{a}}(\boldsymbol{s},\boldsymbol{h})$ is
computed by changing variables as $\boldsymbol{y}=\phi(\boldsymbol{x})$ with
\[
\boldsymbol{y}_{i}=\phi_{i}(\boldsymbol{x}):=\left\{
\begin{array}
[c]{lll}%
\frac{h_{i}\left(  \boldsymbol{a}+\pi\boldsymbol{x}\right)  -h_{i}\left(
\boldsymbol{a}\right)  }{\pi} & \text{if} & i=1,\ldots,m\\
&  & \\
x_{i} & \text{if} & i\geq m+1.
\end{array}
\right.
\]
By using that rank of $Jac(\overline{\boldsymbol{h}}_{I},\overline
{\boldsymbol{a}})$ is $m$ we get that $det\left[  \frac{\partial\phi_{i}%
}{\partial x_{j}}(\mathbf{0})\right]  _{\substack{1\leq i\leq n\\1\leq j\leq
n}}\not \equiv 0\operatorname{mod}\pi$, which implies that $\boldsymbol{y}%
=\phi(\boldsymbol{x})$ gives a measure-preserving map from ${\mathcal{O}}%
_{K}^{n}$ to itself (see e.g. \cite[Lemma 7.4.3]{Igusa}), hence
\[
J_{\boldsymbol{a}}(\boldsymbol{s},\boldsymbol{h})=q^{-n}\prod_{i=1}^{m}%
\int\limits_{{\mathcal{O}}_{K}\backslash\left\{  \pi y_{i}+h_{i}%
(\boldsymbol{a})=0\right\}  }|\pi y_{i}+h_{i}(\boldsymbol{a})|_{K}^{s_{i}%
}|dy_{i}|_{K}=:q^{-n}\prod_{i=1}^{m}J_{\boldsymbol{a}}^{\prime}(y_{i}).
\]
To prove the announced formula we compute integrals $J_{\boldsymbol{a}%
}^{\prime}(y_{i})$. Now, since $h_{i}(\boldsymbol{a})\equiv0\operatorname{mod}%
\pi$, by taking $z_{i}=\pi y_{i}+h_{i}(\boldsymbol{a})$ in $J_{\boldsymbol{a}%
}^{\prime}(y_{i})$, we obtain
\[
J_{\boldsymbol{a}}^{\prime}(y_{i})=q^{-s_{i}}\int\limits_{{\mathcal{O}}%
_{K}\backslash\left\{  0\right\}  }|z_{i}|_{K}^{s_{i}}|dz_{i}|_{K}%
=\frac{(q-1)q^{-1-si}}{1-q^{-1-s_{i}}}.
\]
Therefore
\begin{equation}
J_{\boldsymbol{a}}(\boldsymbol{s},\boldsymbol{h})=%
\begin{cases}
q^{-n} & \text{$I=\varnothing$}\\
q^{-n}\prod_{i\in I}\frac{(q-1)q^{-1-s_{i}}}{1-q^{-1-s_{i}}} & \text{$I\neq
\varnothing$}.
\end{cases}
\label{eqna1}%
\end{equation}

\end{proof}

\begin{remark}
\label{Note1}If in integral (\ref{Integral_J}), we replace $h_{i}%
(\boldsymbol{x})$\ by $h_{i}(\boldsymbol{x})+\pi g_{i}\left(  \boldsymbol{x}%
\right)  $, where each $g_{i}\left(  \boldsymbol{x}\right)  $ is a polynomial
with coefficients in ${\mathcal{O}}_{K}$, then the formulas given in Lemma
\ref{Lemma1} are valid.
\end{remark}

For every subset $I\subseteq\left\{  1,\ldots,r\right\}  $ we set%
\begin{equation}
\overline{V}_{I}:=\left\{  \overline{\boldsymbol{z}}\in(\mathbb{F}_{q}%
^{\times})^{n};\ \overline{h}_{i}(\overline{\boldsymbol{z}})=0\Leftrightarrow
i\in I\right\}  . \label{Notation}%
\end{equation}
To simplify the notation we will denote $\overline{V}_{\left\{  1,\ldots
,r\right\}  }$ as $\overline{V}$.

\begin{lemma}
\label{Lemma3} Let $\boldsymbol{h}=(h_{1},\ldots,h_{r})$ with $r\leq n$, be as
before. Assume that for all $I\neq\varnothing$ if $\overline{V}_{I}%
\neq\varnothing$, then
\[
rank_{\mathbb{F}_{q}}\left[  \frac{\partial\overline{h}_{i}}{\partial x_{j}%
}\left(  \overline{\boldsymbol{a}}\right)  \right]  _{i\in I,\text{ }%
j\in\left\{  1,\ldots,n\right\}  }=Card(I)\text{, for any }\overline
{\boldsymbol{a}}\in\overline{V}_{I}\text{.}%
\]
Set
\[
L(\boldsymbol{s},\boldsymbol{h}):=\int\limits_{({\mathcal{O}}_{K}^{\times
})^{n}\backslash D_{K}}\prod_{i=1}^{r}|h_{i}(\boldsymbol{x})|_{K}^{s_{i}%
}|d\boldsymbol{x}|_{K},\text{ }\boldsymbol{s}=(s_{1},\dots s_{r})\in
\mathbb{C}^{r},
\]
for $\operatorname{Re}\boldsymbol{(}s_{i})>0$ for all $i$. Then, with the
convention that $\prod_{i\in I}\frac{(q-1)q^{-1-s_{i}}}{1-q^{-1-s_{i}}}=1$
when $I=\varnothing$, we have
\[
L(\boldsymbol{s},\boldsymbol{h})=q^{-n}\sum_{I\subseteq\left\{  1,\ldots
,r\right\}  }Card(\overline{V}_{I})\prod_{i\in I}\frac{(q-1)q^{-1-s_{i}}%
}{1-q^{-1-s_{i}}}.
\]

\end{lemma}

\begin{proof}
Note that $L(\boldsymbol{s},\boldsymbol{h})$ can be expressed as a finite sum
of integrals%
\[
J_{\boldsymbol{a}}(\boldsymbol{s},\boldsymbol{h})=\int
\limits_{\substack{\boldsymbol{a}+(\pi{\mathcal{O}}_{K})^{n}\backslash D_{K}%
}}\prod_{i=1}^{r}|h_{i}(\boldsymbol{x})|_{K}^{s_{i}}|d\boldsymbol{x}|_{K},
\]
where $\boldsymbol{a}$ runs through a fixed set of representatives
$\mathcal{R}$ in $\left(  {\mathcal{O}}_{K}^{\times}\right)  ^{n}$ of
$(\mathbb{F}_{q}^{{\mathcal{\times}}})^{n}$. Then $L(\boldsymbol{s}%
,\boldsymbol{h})$ is equals%
\begin{gather*}
\sum_{\overline{\boldsymbol{a}}\in\overline{V}_{\varnothing}}\text{ }%
\int\limits_{\boldsymbol{a}+(\pi{\mathcal{O}}_{K})^{n}\backslash D_{K}}%
\prod_{i=1}^{r}|h_{i}(\boldsymbol{x})|_{K}^{s_{i}}|d\boldsymbol{x}|_{K}\\
+\sum_{\substack{I\subsetneqq\left\{  1,\ldots,r\right\}  \\I\neq\varnothing
}}\text{ }\sum_{\overline{\boldsymbol{a}}\in\overline{V}_{I}}\text{ }%
\int\limits_{\boldsymbol{a}+(\pi{\mathcal{O}}_{K})^{n}\backslash D_{K}}%
\prod_{i=1}^{r}|h_{i}(\boldsymbol{x})|_{K}^{s_{i}}|d\boldsymbol{x}|_{K}\\
+\sum_{\overline{\boldsymbol{a}}\in\overline{V}}\text{ }\int
\limits_{\boldsymbol{a}+(\pi{\mathcal{O}}_{K})^{n}\backslash D_{K}}\text{
}\prod_{i=1}^{r}|h_{i}(\boldsymbol{x})|_{K}^{s_{i}}\left\vert d\boldsymbol{x}%
\right\vert _{K}\\
=:J(\boldsymbol{s},\overline{V}_{\varnothing})+\sum_{\substack{I\subsetneqq
\left\{  1,\ldots,r\right\}  \\I\neq\varnothing}}\text{ }J(\boldsymbol{s}%
,\overline{V}_{I})+J(\boldsymbol{s},\overline{V}),
\end{gather*}
with the convention that if $\overline{V}_{I}=\varnothing$, then
$\sum_{\overline{\boldsymbol{a}}\in\overline{V}_{I}}\int_{\boldsymbol{a}%
+(\pi{\mathcal{O}}_{K})^{n}\backslash D_{K}}\cdot=0$. Notice that%
\begin{equation}
J(\boldsymbol{s},\overline{V}_{\varnothing})=q^{-n}Card(\overline
{V}_{\varnothing}).\label{J_V_O}%
\end{equation}
Thus we may assume that $I\neq\varnothing$. In the calculation of
$J(\boldsymbol{s},\overline{V}_{I})$ we use the following result:

\textbf{Claim. }$\ $\textbf{\ }%
\[
\sum_{\overline{\boldsymbol{a}}\in\overline{V}_{I}}\hspace{0.2cm}%
\int\limits_{\boldsymbol{a}+(\pi{\mathcal{O}}_{K})^{n}\backslash D_{K}}%
\prod_{i=1}^{r}|h_{i}(\boldsymbol{x})|_{K}^{s_{i}}|d\boldsymbol{x}|_{K}%
=\sum_{\overline{\boldsymbol{a}}\in\overline{V}_{I}}\hspace{0.2cm}%
\int\limits_{\substack{\boldsymbol{a}+(\pi{\mathcal{O}}_{K})^{n}\backslash
D_{K}\\\boldsymbol{a}\notin D_{K}}}\prod_{i=1}^{r}|h_{i}(\boldsymbol{x}%
)|_{K}^{s_{i}}|d\boldsymbol{x}|_{K}.
\]
The Claim follows from the following reasoning. The analytic mapping
$h_{1}\cdots h_{r}:\boldsymbol{a}+(\pi{\mathcal{O}}_{K})^{n}\rightarrow K$ is
not identically zero, otherwise by \cite[Lemma 2.1.3]{Igusa}, the polynomial
$(h_{1}\cdots h_{r})(\boldsymbol{x})$ would be the constant polynomial zero
contradicting the hypothesis that all the $h_{i}$'s are non-constant
polynomials. Hence there exists an element $\boldsymbol{b}\in\boldsymbol{a}%
+(\pi{\mathcal{O}}_{K})^{n}$ such that $(h_{1}\cdots h_{r})(\boldsymbol{b}%
)\neq0$. Finally, we use the fact that every point of a ball is its center,
which implies that $\boldsymbol{a}+(\pi{\mathcal{O}}_{K})^{n}=\boldsymbol{b}%
+(\pi{\mathcal{O}}_{K})^{n}$.

By using Lemma \ref{Lemma1},%
\begin{equation}
J(\boldsymbol{s},\overline{V}_{I})=q^{-n}Card(\overline{V}_{I})\prod_{i\in
I}\frac{(q-1)q^{-1-s_{i}}}{1-q^{-1-s_{i}}}. \label{J_V_I}%
\end{equation}
The formula for $J(\boldsymbol{s},\overline{V})$ is a special case of
\ formula (\ref{J_V_I}):
\begin{equation}
J(\boldsymbol{s},\overline{V})=q^{-n}Card(\overline{V})\prod_{i\in\left\{
1,\ldots,r\right\}  }\frac{(q-1)q^{-1-s_{i}}}{1-q^{-1-s_{i}}}. \label{J_V}%
\end{equation}

\end{proof}

\begin{remark}
\label{Note2} In integral $L(\boldsymbol{s},\boldsymbol{h})$ we can replace
$\boldsymbol{h}$ by $\boldsymbol{h}+\pi\boldsymbol{g}$, where $\boldsymbol{g}$
is a polynomial mapping over $\mathcal{O}_{K}$, and the formulas given in
Lemma \ref{Lemma3} remain valid.
\end{remark}

\section{\label{Sect_3}Polyhedral Subdivisions of\textit{\ }$\mathbb{R}%
_{+}^{n}$ and Non-degeneracy conditions}

In this section we review, without proofs, some well-known results about
Newton polyhedra and non-degeneracy conditions that we will use along the
article. Our presentation follows closely \cite{WZ1}, \cite{Oka}.

\subsection{Newton polyhedra}

We set $\mathbb{R}_{+}:=\{x\in\mathbb{R};x\geqslant0\}$. Let $G$ be a
non-empty subset of $\mathbb{N}^{n}$. The \textit{Newton polyhedron }%
$\Gamma=\Gamma\left(  G\right)  $ associated to $G$ is the convex hull in
$\mathbb{R}_{+}^{n}$ \ of the set $\cup_{\boldsymbol{m}\in G}\left(
\boldsymbol{m}+\mathbb{R}_{+}^{n}\right)  $. For instance classically one
associates a \textit{Newton polyhedron }$\Gamma\left(  h\right)
$\textit{\ (at the origin) }to a polynomial function $h(\boldsymbol{x}%
)=\sum_{\boldsymbol{m}}c_{\boldsymbol{m}}\boldsymbol{x}^{\boldsymbol{m}}$
($\boldsymbol{x}=\left(  x_{1},\ldots,x_{n}\right)  $, $h(\boldsymbol{0}%
)=\boldsymbol{0}$), where $G=$supp$(h)$ $:=$ $\left\{  \boldsymbol{m}%
\in\mathbb{N}^{n};c_{\boldsymbol{m}}\neq0\right\}  $.\ Further we will
associate more generally a Newton polyhedron to a polynomial mapping.

We fix a Newton polyhedron $\Gamma$\ as above. We first collect some notions
and results about Newton polyhedra that \ will be used in the next sections.
Let $\left\langle \cdot,\cdot\right\rangle $ denote the usual inner product of
$\mathbb{R}^{n}$, and identify \ the dual space of $\mathbb{R}^{n}$ with
$\mathbb{R}^{n}$ itself by means of it.

Let $H$ be the hyperplane $H=\left\{  \boldsymbol{x}\in{\mathbb{R}}%
^{n};\left\langle \boldsymbol{x},\mathbf{b}\right\rangle =c\right\}  $, $H$
determines two closed half-spaces
\[
H^{+}=\left\{  \boldsymbol{x}\in{\mathbb{R}}^{n};\left\langle \boldsymbol{x}%
,\mathbf{b}\right\rangle \geq c\right\}  \text{ and }H^{-}=\left\{
\boldsymbol{x}\in{\mathbb{R}}^{n};\left\langle \boldsymbol{x},\mathbf{b}%
\right\rangle \leq c\right\}  .
\]

We say that $H$ is a \textit{supporting hyperplane} of $\Gamma(h)$ if
$\Gamma(h)\cap H\neq\varnothing$ and $\Gamma(h)$ is contained in one of the
two closed half-spaces determined by $H$. By a \textit{proper face} $\tau$ of
$\Gamma(h)$, we mean a non-empty convex set $\tau$ obtained by intersecting
$\Gamma(h)$ with one of its supporting hyperplanes. By the \textit{faces} of
$\Gamma(h)$ we will mean the proper faces of $\Gamma(h)$ and the whole the
polyhedron $\Gamma(h)$. By \textit{dimension of a face }$\tau$ of $\Gamma(h)$
we mean the dimension of the affine hull of $\tau$, and its
\textit{codimension} is cod$(\tau)=n-\dim(\tau)$, where $\dim(\tau)$ denotes
the dimension of $\tau$. A face of codimension one is called a \textit{facet}.

For $\boldsymbol{a}\in\mathbb{R}_{+}^{n}$, we define
\[
d(\boldsymbol{a},\Gamma)=\min_{x\in\Gamma}\left\langle \boldsymbol{a}%
,\boldsymbol{x}\right\rangle ,
\]
and \textit{the first meet locus }$F(\boldsymbol{a},\Gamma)$ of
$\boldsymbol{a}$ as \
\[
F(\boldsymbol{a},\Gamma):=\{\boldsymbol{x}\in\Gamma;\left\langle
\boldsymbol{a},\boldsymbol{x}\right\rangle =d(\boldsymbol{a},\Gamma)\}.
\]
The first meet locus \ is a face of $\Gamma$. Moreover, if $\boldsymbol{a}%
\neq\boldsymbol{0}$, $F(\boldsymbol{a},\Gamma)$ is a proper face of $\Gamma$.

If $\Gamma=\Gamma\left(  h\right)  $, we define the \textit{face function
}$h_{\boldsymbol{a}}\left(  \boldsymbol{x}\right)  $\textit{\ of
}$h(\boldsymbol{x})$\textit{\ with respect to }$\boldsymbol{a}$ as
\[
h_{\boldsymbol{a}}\left(  \boldsymbol{x}\right)  =h_{F(\boldsymbol{a},\Gamma
)}\left(  \boldsymbol{x}\right)  =\sum_{\boldsymbol{m}\in F(\boldsymbol{a}%
,\Gamma)}c_{\boldsymbol{m}}\boldsymbol{x}^{\boldsymbol{m}}.
\]

In the case of functions having subindices, say $h_{i}(\boldsymbol{x})$, we
will use the notation $h_{i,\boldsymbol{a}}(\boldsymbol{x})$ for the face
function of $h_{i}(\boldsymbol{x})$\ with respect to $\boldsymbol{a}$. Notice
that $h_{\boldsymbol{0}}\left(  \boldsymbol{x}\right)  =h_{F(\boldsymbol{0}%
,\Gamma)}\left(  \boldsymbol{x}\right)  =\sum_{\boldsymbol{m}}%
c_{\boldsymbol{m}}\boldsymbol{x}^{\boldsymbol{m}}$.

\subsection{Polyhedral Subdivisions Subordinate to a Polyhedron}

We define an equivalence relation in \ $\mathbb{R}_{+}^{n}$ by taking
$\boldsymbol{a}\sim\boldsymbol{a}^{\prime}\Leftrightarrow F(\boldsymbol{a}%
,\Gamma)=F(\boldsymbol{a}^{\prime},\Gamma)$. The equivalence classes of
\ $\sim$ are sets of the form
\[
\Delta_{\tau}=\{\boldsymbol{a}\in\mathbb{R}_{+}^{n};F(\boldsymbol{a}%
,\Gamma)=\tau\},
\]
where $\tau$ \ is a face of $\Gamma$.

We recall that the \textit{cone strictly\ spanned} \ by the vectors
$\boldsymbol{a}_{1},\ldots,\boldsymbol{a}_{l}\in\mathbb{R}_{+}^{n}%
\setminus\left\{  0\right\}  $ is the set $\Delta=\left\{  \lambda
_{1}\boldsymbol{a}_{1}+...+\lambda_{l}\boldsymbol{a}_{l};\lambda_{i}%
\in\mathbb{R}_{+}\text{, }\lambda_{i}>0\right\}  $. If $\boldsymbol{a}%
_{1},\ldots,\boldsymbol{a}_{l}$ are linearly independent over $\mathbb{R}$,
$\Delta$ \ is called \ a \textit{simplicial cone}.\ If \ $\boldsymbol{a}%
_{1},\ldots,\boldsymbol{a}_{l}\in\mathbb{Z}^{n}$, we say $\Delta$\ is \ a
\textit{rational cone}. If $\left\{  \boldsymbol{a}_{1},\ldots,\boldsymbol{a}%
_{l}\right\}  $ is a subset of a basis \ of the $\mathbb{Z}$-module
$\mathbb{Z}^{n}$, we call $\Delta$ a \textit{simple cone}.

A precise description of the geometry of the equivalence classes modulo $\sim$
is as follows. Each \textit{facet}\ $\gamma$ of $\Gamma$\ has a unique vector
$\boldsymbol{a}(\gamma)=(a_{\gamma,1},\ldots,a_{\gamma,n})\in\mathbb{N}%
^{n}\mathbb{\setminus}\left\{  0\right\}  $, \ whose nonzero coordinates are
relatively prime, which is perpendicular to $\gamma$. We denote by
$\mathfrak{D}(\Gamma)$ the set of such vectors. The equivalence classes are
rational cones of the form
\[
\Delta_{\tau}=\{\sum\limits_{i=1}^{r}\lambda_{i}\boldsymbol{a}(\gamma
_{i});\lambda_{i}\in\mathbb{R}_{+}\text{, }\lambda_{i}>0\},
\]
where $\tau$ runs through the set of faces of $\Gamma$, and $\gamma_{i}$,
$i=1,\ldots,r$\ are the facets containing $\tau$. We note that $\Delta_{\tau
}=\{0\}$ if and only if $\tau=\Gamma$. The family $\left\{  \Delta_{\tau
}\right\}  _{\tau}$, with $\tau$ running over\ the proper faces of $\Gamma$,
is a partition of $\mathbb{R}_{+}^{n}\backslash\{0\}$; we call this partition
a \textit{\ polyhedral subdivision of }\ $\mathbb{R}_{+}^{n}$%
\ \textit{subordinate} to $\Gamma$. We call $\left\{  \overline{\Delta}_{\tau
}\right\}  _{\tau}$, the family formed by the topological closures of the
$\Delta_{\tau}$, a \textit{\ fan} \textit{subordinate} to $\Gamma$.

Each cone $\Delta_{\tau}$\ can be partitioned \ into a finite number of
simplicial cones $\Delta_{\tau,i}$. In addition, the subdivision can be chosen
such that each $\Delta_{\tau,i}$ is spanned by part of $\mathfrak{D}(\Gamma)$.
Thus from the above considerations we have the following partition of
$\mathbb{R}_{+}^{n}\backslash\{0\}$:%
\[
\mathbb{R}_{+}^{n}\backslash\{0\}=\bigcup\limits_{\tau\text{ }}\left(
\bigcup\limits_{i=1}^{l_{\tau}}\Delta_{\tau,i}\right)  ,
\]
where $\tau$ runs \ over the proper faces of $\Gamma$, and each $\Delta
_{\tau,i}$ \ is a simplicial cone contained in $\Delta_{\tau}$.\ We will say
that $\left\{  \Delta_{\tau,i}\right\}  $ is a \textit{simplicial polyhedral
subdivision of }\ $\mathbb{R}_{+}^{n}$\ \textit{subordinate} to $\Gamma$, and
that $\left\{  \overline{\Delta}_{\tau,i}\right\}  $ is a \textit{simplicial
fan} \textit{subordinate} to $\Gamma$.

By adding new rays , each simplicial cone can be partitioned further into a
finite number of simple cones. In this way we obtain a \textit{simple
polyhedral subdivision} of $\mathbb{R}_{+}^{n}$\ \textit{subordinate} to
$\Gamma$, and a \textit{simple fan} \textit{subordinate} to $\Gamma$ (or a
\textit{complete regular fan}) (see e.g. \cite{K-M-S}).

\subsection{The Newton polyhedron associated to a polynomial mapping}

Let $\boldsymbol{h}=(h_{1},\ldots,h_{r})$, $\boldsymbol{h}\left(
\boldsymbol{0}\right)  =0$, be a non-constant polynomial mapping. In this
article we associate to $\boldsymbol{h}$ a Newton polyhedron $\Gamma\left(
\boldsymbol{h}\right)  :=\Gamma\left(
{\textstyle\prod\nolimits_{i=1}^{r}}
h_{i}\left(  \boldsymbol{x}\right)  \right)  $. From a geometrical point of
view, $\Gamma\left(  \boldsymbol{h}\right)  $\ is the Minkowski sum of the
$\Gamma\left(  h_{i}\right)  $, for $i=1,\cdots,r$, (see e.g. \cite{Oka},
\cite{Sturmfels}). By using the results previously presented, we can associate
to $\Gamma\left(  \boldsymbol{h}\right)  $ a simplicial polyhedral subdivision
$\mathcal{F}\left(  \boldsymbol{h}\right)  $ of $\mathbb{R}_{+}^{n}$
subordinate to $\Gamma\left(  \boldsymbol{h}\right)  $.

\begin{remark}
\label{remark0}A basic fact\ about the Minkowski sum \ operation is the
additivity of the faces. From this fact follows:

\noindent(1) $F\left(  \boldsymbol{a},\Gamma\left(  \boldsymbol{h}\right)
\right)  =%
{\textstyle\sum\nolimits_{j=1}^{r}}
F\left(  \boldsymbol{a},\Gamma\left(  h_{j}\right)  \right)  $, for
$\boldsymbol{a}\in\mathbb{R}_{+}^{n}$\ ;

\noindent(2) $d\left(  \boldsymbol{a},\Gamma\left(  \boldsymbol{h}\right)
\right)  =%
{\textstyle\sum\nolimits_{j=1}^{r}}
d\left(  \boldsymbol{a},\Gamma\left(  h_{j}\right)  \right)  $, for
$\boldsymbol{a}\in\mathbb{R}_{+}^{n}$\ ;

\noindent(3) let $\tau\ $be a proper face of $\Gamma\left(  \boldsymbol{h}%
\right)  $, and let $\tau_{j}$ be proper face of $\Gamma\left(  h_{j}\right)
$, for $i=1,\cdots,r$. If $\tau=%
{\textstyle\sum\nolimits_{j=1}^{r}}
\tau_{j}$, then $\Delta_{\tau}\subseteq\overline{\Delta}_{\tau_{j}}$, for
$i=1,\cdots,r$.
\end{remark}

\begin{remark}
Note that the equivalence relation,
\[
\boldsymbol{a}\sim\boldsymbol{a}^{\prime}\Leftrightarrow F(a,\Gamma\left(
\boldsymbol{h}\right)  )=F(\boldsymbol{a}^{\prime},\Gamma\left(
\boldsymbol{h}\right)  )\text{,}%
\]
used in the construction of a polyhedral subdivision of $\mathbb{R}_{+}^{n}%
$\ subordinate to\ $\Gamma\left(  \boldsymbol{h}\right)  $ can be equivalently
defined in the following form:%
\[
\boldsymbol{a}\sim\boldsymbol{a}^{\prime}\Leftrightarrow F(\boldsymbol{a}%
,\Gamma\left(  h_{j}\right)  )=F(\boldsymbol{a}^{\prime},\Gamma\left(
h_{j}\right)  )\text{, for each }j=1,\ldots,r.
\]
This last definition is used in Oka's book \cite{Oka}.
\end{remark}

\subsection{Non-degeneracy Conditions}

For $K=%
\mathbb{Q}
_{p},$ Denef and Hoornaert in \cite[Theorem 4.2]{DenefHoor} gave an explicit
formula for $Z(\boldsymbol{s},\boldsymbol{h})$, in the case $r=1$ with
$\boldsymbol{h}$ a non-degenerate polynomial with respect to its Newton
polyhedron $\Gamma(\boldsymbol{h})$. This explicit formula can be generalized
to the case $r\geq1$ by using the condition of non-degeneracy for polynomial
mappings introduced here.

\begin{definition}
\label{Non_Degeneracy_defd}Let $\boldsymbol{h}=(h_{1},\ldots,h_{r})$,
$\boldsymbol{h}\left(  \boldsymbol{0}\right)  =0$, be a polynomial mapping
with $r\leq n$ as in Subsection \ref{Sect_2} and let $\Gamma\left(
\boldsymbol{h}\right)  $ be the Newton polyhedron of $\boldsymbol{h}$ at the
origin. The mapping $\boldsymbol{h}$ is called \textit{non-degenerate over
}$\mathbb{F}_{q}$\textit{ with respect to} $\Gamma\left(  \boldsymbol{h}%
\right)  $, if for every vector $\boldsymbol{k}\in\mathbb{R}_{+}^{n}$ and for
any non-empty subset $I\subseteq\left\{  1,\ldots,r\right\}  $, it verifies
that
\begin{equation}
rank_{\mathbb{F}_{q}}\left[  \frac{\partial\overline{h}_{i,\boldsymbol{k}}%
}{\partial x_{j}}\left(  \overline{\boldsymbol{z}}\right)  \right]  _{i\in
I,\text{ }j\in\left\{  1,\ldots,n\right\}  }=Card(I) \label{Condition_2}%
\end{equation}
for any
\begin{equation}
\overline{\boldsymbol{z}}\in\left\{  \overline{\boldsymbol{z}}\in\left(
\mathbb{F}_{q}^{\times}\right)  ^{n};\overline{h}_{i,\boldsymbol{k}}%
(\overline{\boldsymbol{z}})=0\Leftrightarrow i\in I\right\}  .
\label{Condition_1}%
\end{equation}

\end{definition}

We notice that above notion is different to the those introduced in
\cite{V-Z}, \cite{WZ1}. The notion introduced here is similar to the usual
notion given by Khovansky, see \cite{K}, \cite{Oka}. For a discussion about
the relation between Khovansky's non-degeneracy notion and other similar
notions we refer the reader to \cite{V-Z}.

Let $\Delta$ be a rational simplicial cone spanned by $\boldsymbol{w}_{i}$,
$i=1,\ldots,e_{\Delta}$. We define the \textit{barycenter} of $\Delta$ as
$b(\Delta)=\sum_{i=1}^{e_{\Delta}}\boldsymbol{w}_{i}$. Set $b(\left\{
\mathbf{0}\right\}  ):=\mathbf{0}$.

\begin{remark}
(i)Let $\mathcal{F}(\boldsymbol{h})$ be a simplicial polyhedral subdivision of
$\mathbb{R}_{+}^{n}$ subordinate to $\Gamma\left(  \boldsymbol{h}\right)  $.
Then, it is sufficient to verify the condition given in Definition
\ref{Non_Degeneracy_defd} for $\boldsymbol{k=}b(\Delta)$ with $\Delta
\in\mathcal{F}(\boldsymbol{h})\cup\left\{  \boldsymbol{0}\right\}  $.

(ii) Notice that our notion of non-degeneracy agrees, in the case $K=%
\mathbb{Q}
_{p},$ $r=1$, with the corresponding notion in \cite{DenefHoor}.
\end{remark}

\begin{example}
\label{Ejemplo1}Set $\boldsymbol{h}=(h_{1},h_{2})$ with $h_{1}(x,y)=x^{2}-y$,
$h_{2}(x,y)=x^{2}y$ polynomials in ${\mathcal{O}}_{K}[x,y]$. Then a simplicial
polyhedral subdivision subordinate to $\Gamma(\boldsymbol{h})$ is given by
\end{example}

\begin{center}%
\begin{tabular}
[c]{|l|l|l|}\hline
Cone & $h_{1,b(\Delta)}$ & $h_{2,b(\Delta)}$\\\hline
$\Delta_{1}:=(1,0){\mathbb{R}}_{>0}$ & $y$ & $x^{2}y$\\\hline
$\Delta_{2}:=(1,0){\mathbb{R}}_{>0}+(1,2){\mathbb{R}}_{>0}$ & $y$ & $x^{2}%
y$\\\hline
$\Delta_{3}:=(1,2){\mathbb{R}}_{>0}$ & $x^{2}-y$ & $x^{2}y$\\\hline
$\Delta_{4}:=(1,2){\mathbb{R}}_{>0}+(0,1){\mathbb{R}}_{>0}$ & $x^{2}$ &
$x^{2}y$\\\hline
$\Delta_{5}:=(0,1){\mathbb{R}}_{>0}$ & $x^{2}$ & $x^{2}y,$\\\hline
\end{tabular}

\end{center}

\noindent where ${\mathbb{R}}_{>0}:={\mathbb{R}}_{+}\smallsetminus\left\{
\boldsymbol{0}\right\}  $. Notice that for every $\boldsymbol{k}\in
\mathbb{R}_{+}^{n}\smallsetminus(\left\{  \boldsymbol{0}\right\}  \cup
\Delta_{3})$ and every non-empty subset $I\subseteq\left\{  1,2\right\}  $,
the subset defined in (\ref{Condition_1}) is empty, thus (\ref{Condition_2})
is always satisfied. In the case $\boldsymbol{k}=\boldsymbol{0}$ and
$\boldsymbol{k}\in\Delta_{3}$, $h_{1,\boldsymbol{k}}=x^{2}-y$,
$h_{2,\boldsymbol{k}}=x^{2}y$, the conditions (\ref{Condition_1}%
)-(\ref{Condition_2}) are also verified. Hence $\boldsymbol{h}$ is
non-degenerate over $\mathbb{F}_{q}$ with respect to $\Gamma\left(
\boldsymbol{h}\right)  $.

\begin{example}
Let $\boldsymbol{h}=\left(  h_{1}(\boldsymbol{x}),\ldots,h_{r}(\boldsymbol{x}%
)\right)  $ be a monomial mapping. In this case, $\Gamma\left(  \boldsymbol{h}%
\right)  =\boldsymbol{m}_{0}+{\mathbb{R}}_{+}^{n}$ for some nonzero vector
$\boldsymbol{m}_{0}$ in ${\mathbb{N}}^{n}$. Then for every vector
$\boldsymbol{k}\in\mathbb{R}_{+}^{n}$ \ $h_{i,\boldsymbol{k}}(\boldsymbol{x}%
)=h_{i}(\boldsymbol{x})$ for $i=1,\ldots,r$, and thus the subset in
(\ref{Condition_1}) is always empty, which implies that condition
(\ref{Condition_2}) is always satisfied. Therefore any monomial mapping (with
$r\leq n$) is non-degenerate \textit{over }$\mathbb{F}_{q}$ with respect to
its Newton polyhedron.
\end{example}

\begin{example}
\label{Example_A}$f(\boldsymbol{x}),$ $g(\boldsymbol{x})\in\mathcal{O}%
_{K}[x_{1},...,x_{n}]\backslash\pi\mathcal{O}_{K}[x_{1},...,x_{n}]$ such that
$g(\boldsymbol{x})=\boldsymbol{x}^{\boldsymbol{m}_{0}}$, with $\boldsymbol{m}%
_{0}\neq\boldsymbol{0}$, is a monomial. Suppose that $f$ is non-degenerate
with respect to $\Gamma\left(  f\right)  $ \textit{over }$\mathbb{F}_{q}$. In
this case, $\Gamma\left(  \left(  f,g\right)  \right)  =\boldsymbol{m}%
_{0}+\Gamma\left(  f\right)  $. Then the subset in (\ref{Condition_1}) can
take three different forms:%
\begin{gather*}
\text{(i) }\left\{  \overline{\boldsymbol{z}}\in\left(  \mathbb{F}_{q}%
^{\times}\right)  ^{n};\overline{f}_{\boldsymbol{k}}(\overline{\boldsymbol{z}%
})=\overline{g}\left(  \overline{\boldsymbol{z}}\right)  =0\right\}
=\varnothing\text{, (ii)}\left\{  \overline{\boldsymbol{z}}\in\left(
\mathbb{F}_{q}^{\times}\right)  ^{n};\overline{f}_{\boldsymbol{k}}%
(\overline{\boldsymbol{z}})=0\right\}  \text{,}\\
\text{(iii)}\left\{  \overline{\boldsymbol{z}}\in\left(  \mathbb{F}%
_{q}^{\times}\right)  ^{n};\overline{g}\left(  \overline{\boldsymbol{z}%
}\right)  =0,\overline{f}_{\boldsymbol{k}}(\overline{\boldsymbol{z}}%
)\neq0\right\}  =\varnothing.
\end{gather*}
In the second case, conditions (\ref{Condition_1})-(\ref{Condition_2}) are
verified due to the hypothesis that $f$ is non-degenerate with respect
$\Gamma\left(  f\right)  $ \textit{over }$\mathbb{F}_{q}$. Hence, $(f,g)$ is a
non-degenerate mapping \textit{over }$\mathbb{F}_{q}$ with respect to
$\Gamma\left(  \left(  f,g\right)  \right)  $ \textit{over }$\mathbb{F}_{q}$.
\end{example}

\section{\label{Sect_4}Meromorphic continuation of multivariate local zeta
functions}

Along this section, we work with a fix simplicial polyhedral subdivision
$\mathcal{F}(\boldsymbol{h})$ subordinate to $\Gamma(\boldsymbol{h})$. Let
$\Delta\in\mathcal{F}(\boldsymbol{h})\cup\left\{  \mathbf{0}\right\}  $ and
$I\subseteq\left\{  1,\ldots,r\right\}  $, we put
\[
\overline{V}_{\Delta,I}:=\left\{  \overline{\boldsymbol{z}}\in(\mathbb{F}%
_{q}^{\times})^{n};\ \overline{h}_{i,b(\Delta)}(\overline{\boldsymbol{z}%
})=0\ \Leftrightarrow\ i\in I\right\}  .
\]
We use the convention $\overline{V}_{\Delta,\left\{  1,\ldots,r\right\}
}=\overline{V}_{\Delta}$. If $\Delta=\mathbf{0}$, then
\[
\overline{V}_{\boldsymbol{0},I}=\left\{  \overline{\boldsymbol{z}}%
\in(\mathbb{F}_{q}^{\times})^{n};\ \overline{h}_{i}(\overline{\boldsymbol{z}%
})=0\ \Leftrightarrow\ i\in I\right\}  =\overline{V}_{I},
\]
where $\overline{V}_{I}$ \ is the set defined in (\ref{Notation}). In
particular, $\overline{V}_{\boldsymbol{0},\left\{  1,\ldots,r\right\}
}=\overline{V}$ and
\[
\overline{V}_{\boldsymbol{0},\varnothing}=\left\{  \overline{\boldsymbol{z}%
}\in(\mathbb{F}_{q}^{\times})^{n};\ \overline{h}_{i}(\overline{\boldsymbol{z}%
})\neq0,\ i=1,\ldots,r\right\}  =\overline{V}_{\varnothing}.
\]
If $\boldsymbol{h}=(h_{1},\ldots,h_{r})$ is non-degenerated polynomial mapping
over ${\mathbb{F}}_{q}$ with respect to $\Gamma(\boldsymbol{h})$, then Lemma
\ref{Lemma3} is true for $\boldsymbol{h}_{b(\Delta)}=(h_{1,b(\Delta)}%
,\ldots,h_{r,b(\Delta)})$.

\begin{theorem}
\label{Theorem1} Assume that $\boldsymbol{h}=(h_{1},\ldots,h_{r})$ is
non-degenerated polynomial mapping over ${\mathbb{F}}_{q}$ with respect to
$\Gamma(\boldsymbol{h})$, with $r\leq n$ as before. Fix a simplicial
polyhedral subdivision $\mathcal{F}(\boldsymbol{h})$ subordinate to
$\Gamma(\boldsymbol{h})$. Then $Z(\mathbf{s},\boldsymbol{h})$ has a
meromorphic continuation to $\mathbb{C}^{r}$ as a rational function in the
variables $q^{-s_{i}}$, $i=1,\ldots,r$. In addition, the following explicit
formula holds:
\[
Z(\boldsymbol{s},\boldsymbol{h})=L_{\left\{  \boldsymbol{0}\right\}
}(\boldsymbol{s},\boldsymbol{h})+\sum_{\substack{\text{$\Delta\in
\mathcal{F}(\boldsymbol{h})$}}}L_{\Delta}(\boldsymbol{s},\boldsymbol{h}%
)S_{\Delta},
\]
where
\begin{align*}
L_{\left\{  \mathbf{0}\right\}  }  &  =q^{-n}\sum_{I\subseteq\left\{
1,\ldots,r\right\}  }Card(\overline{V}_{I})\prod_{i\in I}\frac
{(q-1)q^{-1-s_{i}}}{1-q^{-1-s_{i}}},\\
L_{\Delta}  &  =q^{-n}\sum_{I\subseteq\left\{  1,\ldots,r\right\}
}Card(\overline{V}_{\Delta,I})\prod_{i\in I}\frac{(q-1)q^{-1-s_{i}}%
}{1-q^{-1-s_{i}}},
\end{align*}
with the convention that for $I=\varnothing$, $\prod_{i\in I}\frac
{(q-1)q^{-1-s_{i}}}{1-q^{-1-s_{i}}}:=1$, and
\[
S_{\Delta}=\sum_{\boldsymbol{k}\in{\mathbb{N}}^{n}\cap\Delta}q^{-\sigma
(\boldsymbol{k})-\sum_{i=1}^{r}d(\boldsymbol{k},\Gamma(h_{i}))s_{i}}.
\]
Let $\Delta$ be the cone strictly positively generated by linearly independent
vectors $\boldsymbol{w}_{1},\ldots,\allowbreak\boldsymbol{w}_{l}\in
{\mathbb{N}}^{n}\backslash\left\{  \mathbf{0}\right\}  $, then
\[
S_{\Delta}=\frac{\sum_{\boldsymbol{t}}q^{-\sigma(\boldsymbol{t})-\sum
_{i=1}^{r}d(\boldsymbol{t},\Gamma(h_{i}))s_{i}}}{(1-q^{-\sigma(\boldsymbol{w}%
_{1})-\sum_{i=1}^{r}d(\boldsymbol{w}_{1},\Gamma(h_{i}))s_{i}})\cdots
(1-q^{-\sigma(\boldsymbol{w}_{l})-\sum_{i=1}^{r}d(\boldsymbol{w}_{l}%
,\Gamma(h_{i}))s_{i}})},
\]
where $\boldsymbol{t}$ runs through the elements of the set
\begin{equation}
{\mathbb{Z}}^{n}\cap\left\{  \sum_{i=1}^{l}\lambda_{i}\boldsymbol{w}%
_{i};\text{ $0<\lambda_{i}\leq1$ for $i=1,\ldots,l$}\right\}  .
\label{conjunto 1}%
\end{equation}

\end{theorem}

\begin{proof}
By using the simplicial polyhedral subdivision $\mathcal{F}(\boldsymbol{h})$,
we have
\[
{\mathbb{R}}_{+}^{n}=\left\{  \boldsymbol{0}\right\}
{\textstyle\bigsqcup}
{\textstyle\bigsqcup\nolimits_{\Delta\in{\mathcal{F}}(\boldsymbol{h})}}
\Delta.
\]
We set for $\boldsymbol{k}=(k_{1},\ldots,k_{n})\in\mathbb{N}^{n}$,
\[
E_{\boldsymbol{k}}:=\left\{  (x_{1},\ldots,x_{n})\in{\mathcal{O}}_{K}%
^{n};ord(x_{i})=k_{i},i=1,\ldots,n\right\}  .
\]
Hence%
\begin{equation}
Z(\boldsymbol{s},\boldsymbol{h})=\int\limits_{({\mathcal{O}}_{K}^{\times}%
)^{n}\backslash D_{K}}\prod_{i=1}^{r}|h_{i}(x)|_{K}^{s_{i}}|d\boldsymbol{x}%
|_{K}+\sum_{\substack{\text{$\Delta\in\mathcal{F}(\boldsymbol{h})$}}%
}\hspace{0.1cm}\sum_{\boldsymbol{k}\in{\mathbb{N}}^{n}\cap\Delta}%
\hspace{0.3cm}\int\limits_{\substack{E_{\boldsymbol{k}}\backslash D_{K}}%
}\prod_{i=1}^{r}|h_{i}(x)|_{K}^{s_{i}}|d\boldsymbol{x}|_{K}.\nonumber
\end{equation}
For $\Delta\in\mathcal{F}(\boldsymbol{h})$, $\boldsymbol{k}\in{\mathbb{N}}%
^{n}\cap\Delta$, and $\boldsymbol{x}=(x_{1},\ldots,x_{n})\in E_{\boldsymbol{k}%
}$, we put $x_{j}=\pi^{k_{j}}u_{j}$ with $u_{j}\in{\mathcal{O}}_{K}^{\times}$.
Then
\[
|d\boldsymbol{x}|_{K}=q^{-\sigma(\boldsymbol{k})}|d\boldsymbol{u}|_{K}\text{
and }\boldsymbol{x}^{\boldsymbol{m}}=x_{1}^{m_{1}}\cdots x_{n}^{m_{n}}%
=\pi^{\left\langle \boldsymbol{k},\boldsymbol{m}\right\rangle }u_{1}^{m_{1}%
}\cdots u_{n}^{m_{n}}.
\]
Fix $i\in\left\{  1,\ldots,r\right\}  $ and $\boldsymbol{k}$ . For
$\boldsymbol{m}\in supp(h_{i})$, the scalar product $\left\langle
\boldsymbol{k},\boldsymbol{m}\right\rangle $ attains its minimum
$d(\boldsymbol{k},\Gamma(h_{i}))$ exactly when $\boldsymbol{m}\in
F(\boldsymbol{k},\Gamma(h_{i}))$, and thus $\left\langle \boldsymbol{k}%
,\boldsymbol{m}\right\rangle \geq d(\boldsymbol{k},\Gamma(h_{i}))+1$ for
$\boldsymbol{m}\in supp(h_{i})\backslash supp(h_{i})\cap F(\boldsymbol{k}%
,\Gamma(h_{i}))$. This fact implies that
\begin{align}
h_{i}(\boldsymbol{x}) &  =\pi^{d(\boldsymbol{k},\Gamma(h_{i}))}%
(h_{i,\boldsymbol{k}}(\boldsymbol{u})+\pi\widetilde{h}_{i,\boldsymbol{k}%
}(\boldsymbol{u}))\nonumber\\
&  =\pi^{d(\boldsymbol{k},\Gamma(h_{i}))}(h_{i,b(\Delta)}(\boldsymbol{u}%
)+\pi\widetilde{h}_{i,\boldsymbol{k}}(\boldsymbol{u})),\nonumber
\end{align}
where $\widetilde{h}_{i,\boldsymbol{k}}(\boldsymbol{u})$ is a polynomial over
${\mathcal{O}}_{K}$ in the variables $u_{1},\ldots,u_{n}$. Note that
$h_{i,\boldsymbol{k}}(\boldsymbol{u})$ does not depend on $\boldsymbol{k}%
\in\Delta$, for this reason we take $h_{i,\boldsymbol{k}}(\boldsymbol{u}%
)=h_{i,b(\Delta)}(\boldsymbol{u})$. Therefore
\[
Z(\boldsymbol{s},\boldsymbol{h})=L_{\left\{  \boldsymbol{0}\right\}
}(\boldsymbol{s},\boldsymbol{h})+\sum_{\substack{\text{$\Delta\in
\mathcal{F}(\boldsymbol{h})$}}}\hspace{0.1cm}L_{\Delta}(\boldsymbol{s}%
,\boldsymbol{h})\sum_{\boldsymbol{k}\in{\mathbb{N}}^{n}\cap\Delta}%
\hspace{0.3cm}q^{-\sigma(\boldsymbol{k})-\sum_{i=1}^{r}d(\boldsymbol{k}%
,\Gamma(h_{i}))s_{i}}%
\]
where%
\[
L_{\left\{  \boldsymbol{0}\right\}  }(\boldsymbol{s},\boldsymbol{h}%
):=\int\limits_{({\mathcal{O}}_{K}^{\times})^{n}\backslash D_{K}}\prod
_{i=1}^{r}|h_{i}(x)|_{K}^{s_{i}}|d\boldsymbol{x}|_{K},
\]%
\[
L_{\Delta}(\boldsymbol{s},\boldsymbol{h}):=\int\limits_{({\mathcal{O}}%
_{K}^{\times})^{n}\backslash D_{\Delta}}\prod_{i=1}^{r}|h_{i,b(\Delta
)}(\boldsymbol{u})+\pi\widetilde{h}_{i,\boldsymbol{k}}(\boldsymbol{u}%
)|_{K}^{s_{i}}|d\boldsymbol{u}|_{K}%
\]
with $D_{\Delta}=\bigcup_{i=1}^{r}\left\{  \boldsymbol{x}\in({\mathcal{O}}%
_{K}^{\times})^{n};h_{i,b(\Delta)}(\boldsymbol{u})+\pi\widetilde
{h}_{i,\boldsymbol{k}}(\boldsymbol{u})=0\right\}  $. By using the
non-dege\-neracy condition, integrals $L_{\left\{  \boldsymbol{0}\right\}
}(\boldsymbol{s},\boldsymbol{h})$, $L_{\Delta}(\boldsymbol{s},\boldsymbol{h})$
can be computed using Lemma \ref{Lemma3} and Remarks \ref{Note1}, \ref{Note2}.

Let $\Delta$ be the cone strictly positively generated by linearly independent
vectors $\boldsymbol{w}_{1},\ldots,\boldsymbol{w}_{l}\in{\mathbb{N}}%
^{n}\backslash\left\{  \mathbf{0}\right\}  $. If $\Delta$ is a simple cone
then ${\mathbb{N}}^{n}\cap\Delta=\left(  {\mathbb{N}}\backslash\left\{
0\right\}  \right)  \boldsymbol{w}_{1}+\cdots+\left(  {\mathbb{N}}%
\backslash\left\{  0\right\}  \right)  \boldsymbol{w}_{l}$. By using that the
functions $d(\cdot,\Gamma(h_{i})$ are linear over each cone $\Delta$, \ and
that
\[
\sigma(\boldsymbol{w}_{m})+\sum_{i=1}^{r}d(\boldsymbol{w}_{m},\Gamma
(h_{i}))\operatorname{Re}(s_{i})>0,m=1,\ldots,l,
\]
since $\operatorname{Re}(s_{1}),\ldots,\operatorname{Re}(s_{r})>0$, we obtain
\begin{align*}
S_{\Delta}  &  =\sum_{\lambda_{1},\ldots,\lambda_{l}\in{\mathbb{N}}%
\backslash\left\{  0\right\}  }q^{-\sigma(\lambda_{1}\boldsymbol{w}_{1}%
+\ldots+\lambda_{l}\boldsymbol{w}_{l})-\sum_{i=1}^{r}d(\lambda_{1}%
\boldsymbol{w}_{1}+\ldots+\lambda_{l}\boldsymbol{w}_{l},\Gamma(h_{i}))s_{i}}\\
&  =\sum_{\lambda_{1}=1}^{\infty}(q^{-\sigma(\boldsymbol{w}_{1})-\sum
_{i=1}^{r}d(\boldsymbol{w}_{1},\Gamma(h_{i}))s_{i}})^{\lambda_{1}}\cdots
\sum_{\lambda_{l}=1}^{\infty}(q^{-\sigma(\boldsymbol{w}_{l})-\sum_{i=1}%
^{r}d(\boldsymbol{w}_{l},\Gamma(h_{i}))s_{i}})^{\lambda_{l}}\\
S_{\Delta}  &  =\frac{q^{-\sigma(\boldsymbol{w}_{1})-\sum_{i=1}^{r}%
d(\boldsymbol{w}_{1},\Gamma(h_{i}))s_{i}}}{1-q^{-\sigma(\boldsymbol{w}%
_{1})-\sum_{i=1}^{r}d(\boldsymbol{w}_{1},\Gamma(h_{i}))s_{i}}}\cdots
\frac{q^{-\sigma(\boldsymbol{w}_{l})-\sum_{i=1}^{r}d(\boldsymbol{w}_{l}%
,\Gamma(h_{i}))s_{i}}}{1-q^{-\sigma(\boldsymbol{w}_{l})-\sum_{i=1}%
^{r}d(\boldsymbol{w}_{l},\Gamma(h_{i}))s_{i}}}\\
&  =\frac{\sum_{\boldsymbol{t}}q^{-\sigma(\boldsymbol{t})-\sum_{i=1}%
^{r}d(\boldsymbol{t},\Gamma(h_{i}))s_{i}}}{(1-q^{-\sigma(\boldsymbol{w}%
_{1})-\sum_{i=1}^{r}d(\boldsymbol{w}_{1},\Gamma(h_{i}))s_{i}})\cdots
(1-q^{-\sigma(\boldsymbol{w}_{l})-\sum_{i=1}^{r}d(\boldsymbol{w}_{l}%
,\Gamma(h_{i}))s_{i}})},
\end{align*}
where $\boldsymbol{t}$ runs through the elements of the set (\ref{conjunto 1}%
), which consists exactly of one element: $\boldsymbol{t}=\sum_{i=1}%
^{l}\boldsymbol{w}_{i}$. We now consider the case in which $\Delta$ is a
simplicial cone. Note that ${\mathbb{N}}^{n}\cap\Delta$ is the disjoint union
of the sets
\[
\boldsymbol{t}+{\mathbb{N}}\boldsymbol{w}_{1}+\cdots+{\mathbb{N}%
}\boldsymbol{w}_{l},
\]
where $\boldsymbol{t}$ runs through the elements of the set
\[
{\mathbb{Z}}^{n}\cap\left\{  \sum_{i=1}^{l}\lambda_{i}\boldsymbol{w}%
_{i};\text{ $0<\lambda_{i}\leq1$ for $i=1,\ldots,l$}\right\}  .
\]
Hence $S_{\Delta}$ equals
\[
\sum_{\boldsymbol{t}}q^{-\sigma(\boldsymbol{t})-\sum_{i=1}^{r}d(\boldsymbol{t}%
,\Gamma(h_{i}))s_{i}}\sum_{\lambda_{1},\ldots,\lambda_{l}\in{\mathbb{N}}%
}q^{-\sigma(\sum_{j=1}^{l}\lambda_{j}\boldsymbol{w}_{j})-\sum_{i=1}%
^{r}d(\lambda_{1}\boldsymbol{w}_{1}+\ldots+\lambda_{l}\boldsymbol{w}%
_{l},\Gamma(h_{i}))s_{i}},
\]
and since $\operatorname{Re}(s_{1}),\ldots,\operatorname{Re}(s_{r})>0$,
\[
S_{\Delta}=\frac{\sum_{\boldsymbol{t}}q^{-\sigma(\boldsymbol{t})-\sum
_{i=1}^{r}d(\boldsymbol{t},\Gamma(h_{i}))s_{i}}}{(1-{q^{-\sigma(\boldsymbol{w}%
_{1})-\sum_{i=1}^{r}d(\boldsymbol{w}_{1},\Gamma(h_{i}))s_{i}}})\cdots
({1-q^{-\sigma(\boldsymbol{w}_{l})-\sum_{i=1}^{r}d(\boldsymbol{w}_{l}%
,\Gamma(h_{i}))s_{i}}})}.
\]

\end{proof}

\begin{remark}
In the $p$-adic case, $K=%
\mathbb{Q}
_{p},$ Theorem \ref{Theorem1} is a generalization of Theorem 4.2 in
\cite{DenefHoor} and Theorem 4.3 in \cite{Bories}.
\end{remark}

\section{\label{Sect_5}Local zeta function for rational functions}

From now on, we fix two non-constant polynomials $f(\boldsymbol{x})$,
$g(\boldsymbol{x})$ in $n$ variables, $n\geq2$, with coe\-fficients in
$\ \mathcal{O}_{K}[x_{1},\ldots,x_{n}]\backslash\pi{\mathcal{O}}_{K}%
[x_{1},\ldots,x_{n}]$ and \ set $D_{K}:=\left\{  \boldsymbol{x}\in
K^{n};f(\boldsymbol{x})=0\right\}  \cup\left\{  \boldsymbol{x}\in
K^{n};g(\boldsymbol{x})=0\right\}  $, and%
\[
\frac{f}{g}:K^{n}\smallsetminus D_{K}\rightarrow K.
\]
Furthermore, we define the \textit{Newton polyhedron} $\Gamma\left(  \frac
{f}{g}\right)  $ of $\frac{f}{g}$ to be $\Gamma(fg)$, and assume that the
mapping $(f,g):K^{n}\rightarrow K^{2}$ is non-degenerate over $\mathbb{F}_{q}$
with respect to $\Gamma\left(  \frac{f}{g}\right)  $ as before. In this case
we will say that $\frac{f}{g}$ \textit{is non-degenerate over }$\mathbb{F}%
_{q}$\textit{ with respect to} $\Gamma\left(  \frac{f}{g}\right)  $. We fix
\ a simplicial polyhedral subdivision $\mathcal{F}\left(  \frac{f}{g}\right)
$ of $\mathbb{R}_{+}^{n}$ subordinate to $\Gamma\left(  \frac{f}{g}\right)  $.
For $\Delta\in\mathcal{F}\left(  \frac{f}{g}\right)  \cup\left\{
\boldsymbol{0}\right\}  $, we put
\begin{align}
N_{\Delta,\left\{  f\right\}  }  &  :=Card\left\{  \overline{\boldsymbol{a}%
}\in(\mathbb{F}_{q}^{\times})^{n};\overline{\text{$f$}}\text{$_{b(\Delta
)}(\overline{\boldsymbol{a}})=0$ and $\overline{g}_{b(\Delta)}(\overline
{\boldsymbol{a}})\neq0$}\right\}  ,\nonumber\\
N_{\Delta,\left\{  g\right\}  }  &  :=Card\left\{  \overline{\boldsymbol{a}%
}\in(\mathbb{F}_{q}^{\times})^{n};\text{$\overline{f}_{b(\Delta)}%
(\overline{\boldsymbol{a}})\neq0$ and $\overline{g}_{b(\Delta)}(\overline
{\boldsymbol{a}})=0$}\right\}  ,\nonumber\\
N_{\Delta,\left\{  f,g\right\}  }  &  :=Card\left\{  \overline{\boldsymbol{a}%
}\in(\mathbb{F}_{q}^{\times})^{n};\text{$\overline{f}_{b(\Delta)}%
(\overline{\boldsymbol{a}})=0$ and $\overline{g}_{b(\Delta)}(\overline
{\boldsymbol{a}})=0$}\right\}  ,\nonumber
\end{align}
with the convention that if $b(\Delta)=b\left(  \boldsymbol{0}\right)
=\boldsymbol{0}$, then $f$$_{b(\Delta)}=f$ and $g_{b(\Delta)}=g$. We also
define $\mathfrak{D}\left(  \frac{f}{g}\right)  =\mathfrak{D}(f,g)$, which is
the set of primitive vectors in ${\mathbb{N}}^{n}\backslash\left\{
\mathbf{0}\right\}  $ perpendicular to the facets of $\Gamma\left(  \frac
{f}{g}\right)  $. We set
\begin{align}
T_{+}  &  :=\left\{  \boldsymbol{w}\in\mathfrak{D}\left(  \frac{f}{g}\right)
;\text{ $d(\boldsymbol{w},\Gamma(g))-d(\boldsymbol{w},\Gamma(f))>0$}\right\}
,\nonumber\\
T_{-}  &  :=\left\{  \boldsymbol{w}\in\mathfrak{D}\left(  \frac{f}{g}\right)
;\text{ $d(\boldsymbol{w},\Gamma(g))-d(\boldsymbol{w},\Gamma(f))<0$}\right\}
,\nonumber
\end{align}%
\[
\alpha:={\alpha}\left(  \frac{f}{g}\right)  =\left\{
\begin{array}
[c]{lll}%
\min_{\boldsymbol{w}\in T_{+}}\left\{  \frac{\sigma(\boldsymbol{w}%
)}{d(\boldsymbol{w},\Gamma(g))-d(\boldsymbol{w},\Gamma(f))}\right\}  &
\text{if} & \text{$T_{+}\neq\varnothing$}\\
&  & \\
+\infty & \text{if} & \text{$T_{+}=\varnothing,$}%
\end{array}
\right.
\]%
\[
\beta:={\beta}\left(  \frac{f}{g}\right)  =\left\{
\begin{array}
[c]{lll}%
\max_{\boldsymbol{w}\in T_{-}}\left\{  \frac{\sigma(\boldsymbol{w}%
)}{d(\boldsymbol{w},\Gamma(g))-d(\boldsymbol{w},\Gamma(f))}\right\}  &
\text{if} & \text{$T_{-}\neq\varnothing$}\\
&  & \\
-\infty & \text{if} & \text{$T_{-}=\varnothing,$}%
\end{array}
\right.
\]
and%
\[
\widetilde{\alpha}:=\widetilde{\alpha}\left(  \frac{f}{g}\right)
=\min\left\{  1,{\alpha}\right\}  \text{, \ }\widetilde{\beta}:=\widetilde
{\beta}\left(  \frac{f}{g}\right)  =\max\left\{  -1,{\beta}\right\}  .
\]
Notice that ${\alpha}>0$ and ${\beta}<0$.

We define the local zeta function attached to $\frac{f}{g}$ as
\[
Z\left(  s,\frac{f}{g}\right)  =Z(s,-s,f,g),\ s\in\mathbb{C},
\]
where $Z(s_{1},s_{2},f,g)$ denotes the meromorphic continuation of the local
zeta function attached to the polynomial mapping $(f,g)$, see Theorem
\ref{Theorem1}.

\begin{theorem}
\label{Theorem2} Assume that $\frac{f}{g}$ \textit{is non-degenerate over
}$\mathbb{F}_{q}$\textit{ with respect to} $\Gamma\left(  \frac{f}{g}\right)
$, with $n\geq2$ as before. We fix \ a simplicial polyhedral subdivision
$\mathcal{F}\left(  \frac{f}{g}\right)  $ of $\mathbb{R}_{+}^{n}$ subordinate
to $\Gamma\left(  \frac{f}{g}\right)  $. \ Then the following assertions hold:

\noindent(i) $Z\left(  s,\frac{f}{g}\right)  $ has a meromorphic continuation
to the whole complex plane as a rational function of $q^{-s}$ and the
following explicit formula holds:
\[
Z\left(  s,\frac{f}{g}\right)  =\sum_{\text{$\Delta\in\mathcal{F}(\frac{f}%
{g})\cup$}\left\{  \mathbf{0}\right\}  }L_{\Delta}\left(  s,\frac{f}%
{g}\right)  S_{\Delta}(s),
\]
where for $\Delta\in\mathcal{F}\left(  \frac{f}{g}\right)  \cup$$\left\{
\mathbf{0}\right\}  $,%
\begin{align*}
L_{\Delta}(s,\frac{f}{g})  &  =q^{-n}\left[  (q-1)^{n}-N_{\Delta,\left\{
f\right\}  }\frac{1-q^{-s}}{1-q^{-1-s}}-N_{\Delta,\left\{  g\right\}  }%
\frac{1-q^{s}}{1-q^{-1+s}}\right. \\
&  \left.  -N_{\Delta,\left\{  f,g\right\}  }\frac{(1-q^{-s})(1-q^{s}%
)}{q(1-q^{-1-s})(1-q^{-1+s})}\right]
\end{align*}
and
\[
S_{\Delta}(s)=\frac{\sum_{\boldsymbol{t}}q^{-\sigma(\boldsymbol{t}%
)-(d(\boldsymbol{t},\Gamma(f))-d(\boldsymbol{t},\Gamma(g))){s}}}{\prod
_{i=1}^{l}(1-q^{-\sigma(\boldsymbol{w}_{i})-(d(\boldsymbol{w}_{i}%
,\Gamma(f))-d(\boldsymbol{w}_{i},\Gamma(g))){s}})},
\]
\ for $\Delta\in\mathcal{F}\left(  \frac{f}{g}\right)  $ a cone strictly
positively generated by linearly independent vectors $\boldsymbol{w}%
_{1},\ldots,\boldsymbol{w}_{l}\in\mathfrak{D}\left(  \frac{f}{g}\right)  $,
and where $\boldsymbol{t}$ runs through the elements of the set
\[
\mathbb{Z}^{n}\cap\left\{  \sum_{i=1}^{l}\lambda_{i}\boldsymbol{w}_{i};\text{
$0<\lambda_{i}\leq1$ for $i=1,\ldots,$}l\right\}  .
\]
By convention $S_{\boldsymbol{0}}(s):=1$.

\noindent(ii) $Z\left(  s,\frac{f}{g}\right)  $ is a holomorphic function on
$\widetilde{\beta}<\operatorname{Re}(s)<\widetilde{\alpha}$, and on this band
it verifies that
\begin{equation}
Z\left(  s,\frac{f}{g}\right)  =\int\limits_{{\mathcal{O}}_{K}^{n}\backslash
D_{K}}\left\vert \frac{f(x)}{g(x)}\right\vert ^{s}|dx|; \label{integral}%
\end{equation}

\noindent(iii) the poles of the meromorphic continuation of $Z\left(
s,\frac{f}{g}\right)  $ belong to the set
\begin{multline*}
\bigcup\limits_{k\in{\mathbb{Z}}}\left\{  1+\frac{2\pi\sqrt{-1}k}{\ln{q}%
}\right\}  \cup\bigcup\limits_{k\in{\mathbb{Z}}}\left\{  -1+\frac{2\pi
\sqrt{-1}k}{\ln{q}}\right\}  \cup\\
\bigcup\limits_{\boldsymbol{w}\text{$\in\mathfrak{D}$}\left(  \frac{f}%
{g}\right)  }\bigcup\limits_{k\in{\mathbb{Z}}}\left\{  \frac{\sigma
(\boldsymbol{w})}{d(\boldsymbol{w},\Gamma(g))-d(\boldsymbol{w},\Gamma
(f))}+\frac{2\pi\sqrt{-1}k}{\left\{  d(\boldsymbol{w},\Gamma
(g))-d(\boldsymbol{w},\Gamma(f))\right\}  \ln{q}}\right\}  .
\end{multline*}

\end{theorem}

\begin{proof}
(i) The explicit formula for $Z(s,\frac{f}{g})$ follows from Theorem
\ref{Theorem1} as follows: we take $r=2$, $s_{1}=s,$ $s_{2}=-s,$
$h_{1}=f_{b\left(  \Delta\right)  }$ and $h_{2}=g_{b\left(  \Delta\right)  }$
for $\Delta\in\mathcal{F}\left(  \frac{f}{g}\right)  \cup$$\left\{
\mathbf{0}\right\}  $, with the convention that if $b\left(  \Delta\right)
=b\left(  \mathbf{0}\right)  =\mathbf{0}$, then $h_{1}=f$ and $h_{2}=g$. Now
\[
\overline{V}_{\Delta}=\left\{  \overline{\boldsymbol{z}}\in\left(
\mathbb{F}_{q}^{\times}\right)  ^{n};\overline{f}_{b\left(  \Delta\right)
}\left(  \overline{\boldsymbol{z}}\right)  =\overline{g}_{b\left(
\Delta\right)  }\left(  \overline{\boldsymbol{z}}\right)  =0\right\}  \text{
for }\Delta\in\mathcal{F}\left(  \frac{f}{g}\right)  \cup\left\{
\mathbf{0}\right\}  ,
\]
and thus $Card(\overline{V}_{\Delta})=N_{\Delta,\left\{  f,g\right\}  }$. Now,
with $I=\left\{  1,2\right\}  $, by using (\ref{J_V}), we have%
\begin{equation}
J\left(  s,-s,\overline{V}_{\Delta}\right)  =\frac{q^{-n}\left(
1-q^{-1}\right)  ^{2}N_{\Delta,\left\{  f,g\right\}  }}{\left(  1-q^{-1-s}%
\right)  \left(  1-q^{-1+s}\right)  }. \label{J_1}%
\end{equation}
We now consider the case $I\neq\varnothing$, $I\subsetneqq\left\{
1,2\right\}  $, thus there are two cases: $I=\left\{  1\right\}  $ or
$I=\left\{  2\right\}  $. Note that%
\[
\overline{V}_{\Delta,\left\{  1\right\}  }=\left\{  \overline{\boldsymbol{z}%
}\in\left(  \mathbb{F}_{q}^{\times}\right)  ^{n};\overline{f}_{b\left(
\Delta\right)  }\left(  \overline{\boldsymbol{z}}\right)  =0\text{ and
}\overline{g}_{b\left(  \Delta\right)  }\left(  \overline{\boldsymbol{z}%
}\right)  \neq0\right\}  \text{ for }\Delta\in\mathcal{F}\left(  \frac{f}%
{g}\right)  \cup\left\{  \mathbf{0}\right\}  ,
\]
and that $Card\left(  \overline{V}_{\Delta,\left\{  1\right\}  }\right)
=N_{\Delta,\left\{  f\right\}  }$, with the convention that
\[
\overline{V}_{\boldsymbol{0},\left\{  1\right\}  }=\left\{  \overline
{\boldsymbol{z}}\in\left(  \mathbb{F}_{q}^{\times}\right)  ^{n};\overline
{f}\left(  \overline{\boldsymbol{z}}\right)  =0\text{ and }\overline{g}\left(
\overline{\boldsymbol{z}}\right)  \neq0\right\}  \text{.}%
\]
In this case, by using (\ref{J_V_I}),%
\begin{equation}
J\left(  s,-s,\overline{V}_{\Delta,\left\{  1\right\}  }\right)
=\frac{q^{-n-s}\left(  1-q^{-1}\right)  N_{\Delta,\left\{  f\right\}  }%
}{1-q^{-1-s}}. \label{J_2}%
\end{equation}
Analogously,
\begin{equation}
J\left(  s,-s,\overline{V}_{\Delta,\left\{  2\right\}  }\right)
=\frac{q^{-n+s}\left(  1-q^{-1}\right)  N_{\Delta,\left\{  g\right\}  }%
}{1-q^{-1+s}}. \label{J_3}%
\end{equation}
We now consider the case $I=\varnothing$, then
\[
\overline{V}_{\Delta,\varnothing}=\left\{  \overline{\boldsymbol{z}}\in\left(
\mathbb{F}_{q}^{\times}\right)  ^{n};\overline{f}_{b\left(  \Delta\right)
}\left(  \overline{\boldsymbol{z}}\right)  \neq0\text{ and }\overline
{g}_{b\left(  \Delta\right)  }\left(  \overline{\boldsymbol{z}}\right)
\neq0\right\}  \text{ for }\Delta\in\mathcal{F}\left(  \frac{f}{g}\right)
\cup\left\{  \mathbf{0}\right\}  ,
\]
with the convention that
\[
\overline{V}_{\boldsymbol{0},\varnothing}=\left\{  \overline{\boldsymbol{z}%
}\in\left(  \mathbb{F}_{q}^{\times}\right)  ^{n};\overline{f}\left(
\overline{\boldsymbol{z}}\right)  \neq0\text{ and }\overline{g}\left(
\overline{\boldsymbol{z}}\right)  \neq0\right\}  .
\]
Notice that $Card(\overline{V}_{\Delta,\varnothing})=\left(  q-1\right)
^{n}-N_{\Delta,\left\{  f\right\}  }-N_{\Delta,\left\{  g\right\}  }%
-N_{\Delta,\left\{  f,g\right\}  }$. Then, by using (\ref{J_V_O}),%
\begin{equation}
J\left(  s,-s,\overline{V}_{\Delta,\varnothing}\right)  =q^{-n}Card(\overline
{V}_{\Delta,\varnothing}). \label{J_4}%
\end{equation}
Then from Theorem \ref{Theorem1} and (\ref{J_1})-(\ref{J_4}), we get%
\begin{multline*}
L_{\Delta}(s,\frac{f}{g})=\frac{q^{-n}\left(  1-q^{-1}\right)  ^{2}%
N_{\Delta,\left\{  f,g\right\}  }}{\left(  1-q^{-1-s}\right)  \left(
1-q^{-1+s}\right)  }+\frac{q^{-n-s}\left(  1-q^{-1}\right)  N_{\Delta,\left\{
f\right\}  }}{1-q^{-1-s}}+\\
\frac{q^{-n+s}\left(  1-q^{-1}\right)  N_{\Delta,\left\{  g\right\}  }%
}{1-q^{-1+s}}+q^{-n}\left\{  \left(  q-1\right)  ^{n}-N_{\Delta,\left\{
f\right\}  }-N_{\Delta,\left\{  g\right\}  }-N_{\Delta,\left\{  f,g\right\}
}\right\}  .
\end{multline*}
The announced formula for $L_{\Delta}(s,\frac{f}{g})$ is obtained from the
above formula after some simple algebraic manipulations.

(ii) Notice that for $\boldsymbol{w}\in\mathfrak{D}\left(  \frac{f}{g}\right)
$, $\frac{1}{1-q^{-\sigma(\boldsymbol{w})-(d(\boldsymbol{w},\Gamma
(f))-d(\boldsymbol{w},\Gamma(g))){s}}}$ is holomorphic on $\sigma
(\boldsymbol{w})+(d(\boldsymbol{w},\Gamma(f))-d(\boldsymbol{w},\Gamma
(g)))\operatorname{Re}(s)>0$, and that $\frac{1}{1-q^{-1-s}}$ is holomorphic
on $\operatorname{Re}(s)>-1$, and $\frac{1}{1-q^{-1+s}}$ is holomorphic on
$\operatorname{Re}(s)<1$, then, from the explicit formula for $Z(s,\frac{f}%
{g})$ given in (i) follows that it is holomorphic on the band $\widetilde
{\beta}<\operatorname{Re}(s)<\widetilde{\alpha}$. Now, since $Z(s,\frac{f}%
{g})=Z(s,-s,f,g)$, $Z(s,\frac{f}{g})$ is given by integral (\ref{integral})
because $Z(s_{1},s_{2},f,g)$ agrees with an integral on its natural domain.

(iii) It is a direct consequence of the explicit formula.
\end{proof}

\section{\label{Sect_6}The largest and smallest real part of the poles of
$Z(s,\frac{f}{g})$ (different from $-1$ and $1$, respectively)}

In this section we use all the notation introduced in Section \ref{Sect_5}. We
work with a fix simplicial polyhedral subdivision $\mathcal{F}\left(  \frac
{f}{g}\right)  $ of $\mathbb{R}_{+}^{n}$ subordinate to $\Gamma\left(
\frac{f}{g}\right)  $. We recall that in the case $T_{-}\neq\varnothing$,
\[
\beta=\max_{\boldsymbol{w}\in T_{-}}\left\{  \frac{\sigma(\boldsymbol{w}%
)}{d(\boldsymbol{w},\Gamma(g))-d(\boldsymbol{w},\Gamma(f))}\right\}
\]
is the largest possible `non-trivial' negative real part of the poles of
$Z(s,\frac{f}{g})$. We set
\[
\mathcal{P}(\beta):=\left\{  \boldsymbol{w}\in T_{-};\frac{\sigma
(\boldsymbol{w})}{d(\boldsymbol{w},\Gamma(g))-d(\boldsymbol{w},\Gamma
(f))}=\beta\right\}  ,
\]
and for $m\in\mathbb{N}$ with $1\leq m\leq n$,
\[
\mathcal{M}_{m}(\beta):=\left\{  \text{$\Delta\in\mathcal{F}\left(  \frac
{f}{g}\right)  $; $\Delta$ has exactly $m$ generators belonging to
}\mathcal{P}(\beta)\right\}  ,
\]
and $\rho:=\max\left\{  m;\mathcal{M}_{m}(\beta)\neq\varnothing\right\}  $.

\begin{theorem}
\label{Theorem3}Suppose that $\frac{f}{g}$ is non-degenerated over
${\mathbb{F}}_{q}$ with respect to $\Gamma(\frac{f}{g})$ and that $T_{-}%
\neq\varnothing$. If $\beta>-1$, then $\beta$ is a pole of $Z(s,\frac{f}{g})$
of multiplicity $\rho$.
\end{theorem}

\begin{proof}
In order to prove that $\beta$ is a pole of $Z(s,\frac{f}{g})$ of order $\rho
$, it is sufficient to show that
\[
\lim_{s\rightarrow\beta}(1-q^{\beta-s})^{\rho}Z\left(  s,\frac{f}{g}\right)
>0.
\]
This assertion follows from the explicit formula for $Z(s,\frac{f}{g})$ given
in Theorem \ref{Theorem2}, by the following claim:

\textbf{Claim. }$Res\left(  \Delta,\beta\right)  :=\lim_{s\rightarrow\beta
}(1-q^{s-\beta})^{\rho}L_{\Delta}(s,\frac{f}{g})S_{\Delta}(s)\geq0$\textbf{
}for every cone $\Delta\in\mathcal{F}(\frac{f}{g})$. Furthermore, there exists
a cone $\Delta_{0}\in\mathcal{M}_{\rho}(\beta)$ such that $Res\left(
\Delta_{0},\beta\right)  \allowbreak>0$.

We show that\ for at least one cone $\Delta_{0}$ in $\mathcal{M}_{\rho}%
(\beta)$, $Res\left(  \Delta_{0},\beta\right)  \allowbreak>0$, because for any
cone $\Delta\notin\mathcal{M}_{\rho}(\beta)$, $Res\left(  \Delta,\beta\right)
=0$. This last assertion can be verified by using the argument that we give
for the cones in $\mathcal{M}_{\rho}(\beta)$. We first note that there exists
at least one cone $\Delta_{0}$ in $\mathcal{M}_{\rho}(\beta)$. Let
$\boldsymbol{w}_{1},\ldots,\boldsymbol{w}_{\rho},\allowbreak\boldsymbol{w}%
_{\rho+1},\ldots,\boldsymbol{w}_{l}$ its generators with $\boldsymbol{w}%
_{i}\in\mathcal{P}(\beta)\Leftrightarrow1\leq i\leq\rho$.

On the other hand,
\begin{equation}
\lim_{s\rightarrow\beta}L_{\Delta}\left(  s,\frac{f}{g}\right)  >0
\label{calculo de residuo2}%
\end{equation}
for all cones $\Delta\in\mathcal{F}(\frac{f}{g})\cup\left\{  \mathbf{0}%
\right\}  $. Inequality (\ref{calculo de residuo2}) follows from
\[
L_{\Delta}\left(  \beta,\frac{f}{g}\right)  >q^{-n}((q-1)^{n}-N_{\Delta
,\left\{  f\right\}  }-N_{\Delta,\left\{  g\right\}  }-N_{\Delta,\left\{
f,g\right\}  })\geq0
\]
for all cones $\Delta\in\mathcal{F}(\frac{f}{g})\cup\left\{  \mathbf{0}%
\right\}  $. We prove this last inequality in the case $N_{\Delta,\left\{
f\right\}  }>0$, $N_{\Delta,\left\{  g\right\}  }>0$, $N_{\Delta,\left\{
f,g\right\}  }>0$ since the other cases are treated in similar form. In this
case, the inequality \ follows from the formula for $L_{\Delta}(\beta,\frac
{f}{g})$ given in Theorem \ref{Theorem2} , by using that
\begin{multline*}
N_{\Delta,\left\{  f\right\}  }\frac{1-q^{-\beta}}{1-q^{-1-\beta}}%
<N_{\Delta,\left\{  f\right\}  },\text{ }N_{\Delta,\left\{  g\right\}  }%
\frac{1-q^{\beta}}{1-q^{-1+\beta}}<N_{\Delta,\left\{  g\right\}  },\text{ }\\
N_{\Delta,\left\{  f,g\right\}  }\frac{(1-q^{-\beta})(1-q^{\beta}%
)}{q(1-q^{-1-\beta})(1-q^{-1+\beta})}<N_{\Delta,\left\{  f,g\right\}  }\text{
\ when }\beta>-1\text{.}%
\end{multline*}

We also notice that
\[
\lim_{s\rightarrow\beta}\sum_{\boldsymbol{t}}q^{-\sigma(\boldsymbol{t}%
)-(d(\boldsymbol{t},\Gamma(f))-d(\boldsymbol{t},\Gamma(g))){s}}>0.
\]
Hence in order to show that $Res\left(  \Delta_{0},\beta\right)
\allowbreak>0$, it is sufficient to show that
\[
\lim_{s\rightarrow\beta}\frac{(1-q^{s-\beta})^{\rho}}{\prod_{i=1}%
^{l}(1-q^{-\sigma(\boldsymbol{w}_{i})-(d(\boldsymbol{w}_{i},\Gamma
(f))-d(\boldsymbol{w}_{i},\Gamma(g))){s}})}>0.
\]
Now, notice that there are positive integer constants $c_{i}$ such that
\begin{gather*}
\prod_{i=1}^{\rho}(1-q^{-\sigma(\boldsymbol{w}_{i})-(d(\boldsymbol{w}%
_{i},\Gamma(f))-d(\boldsymbol{w}_{i},\Gamma(g))){s}})=\prod_{i=1}^{\rho
}(1-q^{\left(  {s}-\beta\right)  c_{i}})\\
=(1-q^{s-\beta})^{\rho}\prod_{i=1}^{\rho}\text{ }\prod\limits_{\varsigma
^{c_{i}}=1,\varsigma\neq1}\left(  1-\varsigma q^{s-\beta}\right)  .
\end{gather*}
In addition, for $i=\rho+1,\ldots,l$,
\[
1-q^{-\sigma(\boldsymbol{w}_{i})-(d(\boldsymbol{w}_{i},\Gamma
(f))-d(\boldsymbol{w}_{i},\Gamma(g))){\beta}}>0
\]
because $-\sigma(\boldsymbol{w}_{i})-(d(\boldsymbol{w}_{i},\Gamma
(f))-d(\boldsymbol{w}_{i},\Gamma(g))){\beta\leq0}$ for any $\boldsymbol{w}%
_{i}\in T_{+}\cup T_{-}$ with $i=\rho+1,\ldots,l$. From these observations, we
have%
\begin{gather*}
\lim_{s\rightarrow\beta}\frac{(1-q^{s-\beta})^{\rho}}{\prod_{i=1}%
^{l}(1-q^{-\sigma(\boldsymbol{w}_{i})-(d(\boldsymbol{w}_{i},\Gamma
(f))-d(\boldsymbol{w}_{i},\Gamma(g))){s}})}=\\
\lim_{s\rightarrow\beta}\frac{(1-q^{s-\beta})^{\rho}}{(1-q^{s-\beta})^{\rho
}\prod_{i=1}^{\rho}\text{ }\prod\limits_{\varsigma^{c_{i}}=1,\varsigma\neq
1}\left(  1-\varsigma q^{s-\beta}\right)  }\times\\
\lim_{s\rightarrow\beta}\frac{1}{\prod_{i=\rho+1}^{l}(1-q^{-\sigma
(\boldsymbol{w}_{i})-(d(\boldsymbol{w}_{i},\Gamma(f))-d(\boldsymbol{w}%
_{i},\Gamma(g))){s}})}>0.
\end{gather*}

\end{proof}

In the case $T_{+}\neq\varnothing$,
\[
{\alpha}=\min_{\boldsymbol{w}\in T_{+}}\left\{  \frac{\sigma(\boldsymbol{w}%
)}{d(\boldsymbol{w},\Gamma(g))-d(\boldsymbol{w},\Gamma(f))}\right\}  .
\]
is the smallest possible `non-trivial' positive real part of the poles of
$Z(s,\frac{f}{g})$. We set
\[
\mathcal{P}(\alpha):=\left\{  \boldsymbol{w}\in T_{+};\frac{\sigma
(\boldsymbol{w})}{d(\boldsymbol{w},\Gamma(g))-d(\boldsymbol{w},\Gamma
(f))}=\alpha\right\}  ,
\]
and for $m\in\mathbb{N}$ with $1\leq m\leq n$,
\[
\mathcal{M}_{m}(\alpha):=\left\{  \text{$\Delta\in\mathcal{F}\left(  \frac
{f}{g}\right)  $; $\Delta$ has exactly $m$ generators belonging to
}\mathcal{P}(\alpha)\right\}  ,
\]
and $\kappa:=\max\left\{  m;\mathcal{M}_{m}(\alpha)\neq\varnothing\right\}  $

The proof of the following result is similar to the proof of Theorem
\ref{Theorem3}.

\begin{theorem}
\label{Theorem4}Suppose that $\frac{f}{g}$ is non-degenerated over
${\mathbb{F}}_{q}$ with respect to $\Gamma(\frac{f}{g})$ and that $T_{+}%
\neq\varnothing$. If $\alpha<1$, then $\alpha$ is a pole of $Z(s,\frac{f}{g})$
of multiplicity $\kappa$.
\end{theorem}

\begin{example}
\label{Ej1}We compute the local zeta function for the rational function given
in Example \ref{Ejemplo1}. With the notation of Theorem \ref{Theorem2}, one
verifies that%
\[
\overset{%
\begin{tabular}
[c]{|l|c|c|}\hline
$Cone$ & $L_{\Delta}$ & $S_{\Delta}$\\\hline
\multicolumn{1}{|c|}{$\left\{  \boldsymbol{0}\right\}  $} & $q^{-2}%
((q-1)^{2}-(q-1)\frac{1-q^{-s}}{1-q^{-1-s}})$ & $1$\\\hline
\multicolumn{1}{|c|}{$\Delta_{1}$} & $q^{-2}(q-1)^{2}$ & $\frac{q^{-1+2s}%
}{1-q^{-1+2s}}$\\\hline
\multicolumn{1}{|c|}{$\Delta_{2}$} & $q^{-2}(q-1)^{2}$ & $\frac{q^{-2+2s}%
+q^{-4+4s}}{(1-q^{-1+2s})(1-q^{-3+2s})}$\\\hline
\multicolumn{1}{|c|}{$\Delta_{3}$} & $q^{-2}((q-1)^{2}-(q-1)\frac{1-q^{-s}%
}{1-q^{-1-s}})$ & $\frac{q^{-3+2s}}{1-q^{-3+2s}}$\\\hline
\multicolumn{1}{|c|}{$\Delta_{4}$} & $q^{-2}(q-1)^{2}$ & $\frac{q^{-4+3s}%
}{(1-q^{-3+2s})(1-q^{-1+s})}$\\\hline
\multicolumn{1}{|c|}{$\Delta_{5}$} & $q^{-2}(q-1)^{2}$ & $\frac{q^{-1+s}%
}{(1-q^{-1+s})}.$\\\hline
\end{tabular}
\ \ \ \ \ \ \ \ }{}%
\]
Therefore%
\[
Z(s,\frac{f}{g})=\frac{\frac{\left(  q-1\right)  }{q^{2}}L(q^{-s})}%
{(1-q^{s-1})(1-q^{-1-s})(1-q^{2s-1})(1-q^{2s-3})},
\]
where%
\begin{align*}
L(q^{-s})  & =q-q^{-1}-2-q^{2s-4}+q^{s-3}-q^{s-2}+q^{2s-2}+q^{3s-3}\\
& +2q^{2s-1}-q^{3s-2}-q^{3s-1}+q^{-s-1}.
\end{align*}
\ Furthermore, $Z(s,\frac{f}{g})$ has poles with real parts belonging to
$\left\{  -1,1/2,1,3/2\right\}  $.
\end{example}

\end{document}